\documentclass[10pt,fleqn,final]{siamltex}
\usepackage{amsmath,amssymb,latexsym}

\newtheorem{prop}{Proposition}
\newtheorem{cor}{Corollary}

\newcommand\bu {{\bf u}}

\newcommand{\lb}{\lambda_1}
\newcommand{\fvi}{\|f\|_{\infty}}

\newcommand{\om}{\omega}

\newcommand{\Real}{\mathbf{R}}

\newcommand{\dy}{\partial}
\newcommand{\gb}{\nabla}
\newcommand{\sgb}{\gb^\perp}
\newcommand{\bb}{b}
\newcommand{\vfi}{\varphi}

\newcommand\Hperz{\mathring{\mathrm{H}}_{\textrm{per}}}
\newcommand\Omt{\tilde\Omega}

\newcommand{\Lr}{\dot{L}}
\newcommand{\Hr}{\dot{H}}
\newcommand{\sfrac}[2]{{\textstyle\frac{#1}{#2}}}
\newcommand{\dtt}{\;\mathrm{d}t}

\newcommand{\omegaz}{\omega_0}
\newcommand{\Attr}{\mathcal{A}}

\def\DOT{\!\cdot\!}
\def\dt{{\Delta t}}

\def\u{\mbox{\boldmath $u$}}
\def\U{\mbox{\boldmath $U$}}

\def\f{\mbox{\boldmath $f$}}

\def\x{\mbox{\boldmath $x$}}
\def\pomega{\mbox{\boldmath $\omega$}}
\def\ppsi{\mbox{\boldmath $\psi$}}

\title{Long time stability of a classical efficient scheme for two dimensional Navier--Stokes equations}

\author{%
S.~Gottlieb\thanks{Department of Mathematics, University of Massachusetts Dartmouth, North Dartmouth, MA 02747 ({\tt sgottlieb@umassd.edu})}
\and
F.~Tone\thanks{Department of Mathematics and Statistics, University of West Florida, Pensacola, FL 32514 ({\tt ftone@uwf.edu})}
\and
C.~Wang\thanks{Department of Mathematics, University of Massachusetts Dartmouth, North Dartmouth, MA 02747 ({\tt cwang1@umassd.edu})}
\and
X.~Wang\thanks{Department of Mathematics, Florida State University, Tallahassee, FL 32306--4510\hfill\break
({\tt wxm@math.fsu.edu})}
\and
D.~Wirosoetisno\thanks{Department of Mathematical Sciences, Durham University, Durham\ \ DH1~3LE, United Kingdom ({\tt djoko.wirosoetisno@durham.ac.uk})}
}

\begin{document}

\maketitle

\begin{abstract}
We prove that a popular classical implicit-explicit scheme for the 2D incompressible Navier--Stokes equations that treats the viscous term implicitly while the nonlinear advection term explicitly is long time stable provided that the time step is sufficiently small in the case with periodic boundary conditions.
The long time stability in the $L^2$ and $H^1$ norms further leads to the convergence of the global attractors and invariant measures of the scheme to those of the NSE itself at vanishing time step. Both semi-discrete in time and fully discrete schemes with either Galerkin Fourier spectral or collocation Fourier spectral methods are considered.
\end{abstract}

\begin{keywords}
2d Navier--Stokes equations,
semi-implicit schemes,
global attractor,
invariant measures,
spectral and collocation
\end{keywords}

\begin{AMS}
65M12, 
65M70, 
76D06, 
37L40  
\end{AMS}

\pagestyle{myheadings}
\thispagestyle{plain}
\markboth{S.~GOTTLIEB ET AL.}{A CLASSICAL SCHEME FOR 2D NAVIER--STOKES}

\section{Introduction}

The celebrated Navier--Stokes system for homogeneous incompressible Newtonian fluids in the  vorticity--streamfunction formulation in two dimensions takes the form
\begin{equation}\label{NSE}\begin{aligned}
   \frac{\partial\omega}{\partial t}+\nabla^\perp\psi\cdot\nabla\omega -\nu\Delta \omega &= f,\\
   -\Delta\psi &= \omega,
\end{aligned}\end{equation}
where  $\omega$ denotes the vorticity, $\psi$ is the streamfunction, and $f$ represents (given) external forcing.  For simplicity we will assume periodic boundary condition, i.e., the domain is a two dimensional torus $\mathbb{T}^2$, and that all functions have mean zero over the torus.

It is well-known that two dimensional incompressible flows could be extremely complicated with possible chaos and turbulent behavior \cite{F1995, FMRT2001, MY1975, CF1988, MW2006, T1997}.
Although some of the features of this turbulent or chaotic behavior may be deduced via analytic means, it is widely believed that numerical methods are indispensable for obtaining a better understanding of these complicated phenomena.
For analytic forcing, it is known that the solution is analytic in space (in fact Gevrey class regular \cite{FT1989}), and hence Fourier spectral is the obvious choice for spatial discretization.
As for time discretization, one of the popular schemes \cite{CHQZ1988, Peyret2002} is the following semi-implicit algorithm, which treats the viscous term implicitly and the nonlinear advection term explicitly

\begin{equation}\label{scheme}
    \frac{\omega^{n+1}-\omega^n}{\dt} + \nabla^\perp\psi^n\cdot\nabla\omega^n - \nu\Delta\omega^{n+1} = f^n.
\end{equation}
ere $\dt$ is the time step, and $\omega^n, \omega^{n+1}$ are the approximations of the vorticity at  the discrete times $n\dt, (n+1)\dt$, respectively.  The convergence of this scheme on any fixed time interval is standard and well-known \cite{G1991a, G1991b, G1992, GS1998, TM1996}. There are many off-the-shelf efficient solvers of (\ref{scheme}), since it essentially reduces to a Poisson solver  at each time step.

   It is also well-known that the NSE (\ref{NSE}) is long time enstrophy stable in the sense that the enstrophy $\bigl(\sfrac12\|\omega\|_{L^2}^2\bigr)$ is bounded uniformly in time, and it possesses a global attractor $\Attr$ and invariant measures \cite{CF1988, FMRT2001, T1997}. In fact, it is the long time dynamics characterized by the global attractor and invariant measure that are central to the understanding of turbulence. Therefore a natural question  is if numerical schemes such as (\ref{scheme}) can capture the long time dynamics of the NSE (\ref{NSE}) in the sense of convergence of global attractors and invariant measures. To say the least, we would require that the scheme inherit the long time stability of the NSE.

   There is a long  list of works on time discretization of the NSE and related dissipative systems that preserve the dissipativity in various forms \cite{S1989, S1990, FJKT1991, FJKT1994, HS2000, SH1996, J2002, T2009,TW2006}. It has also been discovered recently that if the dissipativity of a dissipative system is preserved appropriately, then the numerical scheme would be able to capture the long time statistical property of the underlying dissipative system asymptotically, in the sense that the invariant measures of the scheme would converge to those of the continuous-in-time system \cite{W2010}. The main purpose of this manuscript is to show that the classical scheme (\ref{scheme}) is long time stable in $L^2$ and $H^1$, and that the global attractor as well as the invariant measures of the scheme, converge to those of the NSE at vanishing time step.

\section{Long time behavior of the semi-discrete scheme}

We first recall the well-known periodic Sobolev spaces on $\Omega=(0,2\pi)\times (0,2\pi)$ with average zero:
\begin{equation}
  \Hr^m_{per}(\Omega) := \left\{ \phi \in H^m(\Omega) \bigg| \int_\Omega \phi = 0 \textrm{ and }\phi\textrm{ is } 2\pi\text{-periodic in each direction} \right\}.
\end{equation}
$\Hr^{-m}_{per}$ is defined as the dual space of $ \Hr^m_{per}$ with the duality induced by the $L^2$ inner product.
The adoption of $ \Hr^m_{per}$ is well-known \cite{CF1988, T1983} since this space is invariant under the Navier--Stokes dynamics (\ref{NSE}), provided that the initial data and the forcing term belong to the same space.

\subsection{Long time stability of the scheme}
We first prove that the scheme \eqref{scheme} is stable for all time.

\medskip
\begin{lemma}\label{dynsyst}
The scheme \eqref{scheme} forms a dynamical system on $\Lr^2$.
\end{lemma}

\medskip
\begin{proof}
It is easy to see that for $\omega^n \in \Lr^2$, we have $\psi^n \in \Hr^2_{per}$.
Hence $\nabla^\perp\psi^n\cdot\nabla\omega^n\in \Hr^{-1-\alpha}_{per}$ for all $\alpha\in (0,1)$.
Therefore, the classical scheme (\ref{scheme}), which can be viewed as a Poisson type problem
$\omega^{n+1}/\dt - \nu\Delta\omega^{n+1} =
f -\nabla^\perp\psi^n\cdot\nabla\omega^n  + \omega^n/\dt
\in \Hr^{-1-\alpha}_{per}$,
possesses a unique solution in $\Lr^2$ (in fact in $\Hr^1_{per}$)  and the solution depends continuously on the data.
Therefore it defines a (discrete) semi-group on $\Lr^2$.
\qquad\end{proof}

\medskip
  Now we derive the long time stability of the scheme (\ref{scheme}) both in $L^2$ and in $H^1$. Our proof relies on a Wente type estimate on the
nonlinear term (see Appendix~\ref{s:wente}), which may be of independent interest.

We first show that  the scheme (\ref{scheme})  is uniformly bounded
in $L^2$, provided that the time step is sufficiently small.
To this end,
we take the scalar product of (\ref{scheme}) with $2\dt\,\omega^{n+1}$
and using the relation
\begin{equation}
      2(\varphi - \psi, \varphi)_{L^2}=\|\varphi\|_2^2-\|\psi\|_2^2+\|\varphi-\psi\|_2^2,
\end{equation}
where $\|\cdot\|_2^{}$ denotes the $L^2$ norm, we obtain
\begin{equation}\label{1}\begin{aligned}
     \|\om^{n+1}\|_2^2 -\|\om^{n}\|_2^2 + \|\omega^{n+1} - \omega ^{n}\|_2^2
         + 2 \nu \dt \| \omega ^{n+1}\|^2_{H^1} &+ 2\dt\, \bb(\psi^n, \omega^n, \omega^{n+1})\\
	&= 2  \dt\, (f^n, \omega ^{n+1})_{L^2}
\end{aligned}\end{equation}
where
\begin{equation}
   b(\psi,\omega,\tilde\omega) := (\sgb\psi\cdot\gb\omega,\tilde\omega)_{L^2}^{}
	= -b(\psi,\tilde\omega,\omega),
\end{equation}
the last equality obtaining upon integration by parts.
Using the Cauchy--Schwarz and the Poincar\'e inequalities, we majorize the right-hand side of (\ref{1}) by
\begin{equation}\label{2}
      2\dt \|f^n\|_2 \|\omega ^{n+1}\|_2
       \leq 2c_0\dt \|f^n\|_2 \|\omega ^{n+1}\|_{H^1}
       \leq \nu \dt \|\omega ^{n+1}\|^2_{H^1} + \frac{c_0^2}{\nu }\dt \|f^n\|^2_2.
\end{equation}
Using the Wente type estimate (\ref{q:wente2}), we bound the nonlinear term as
\begin{equation}\label{3}\begin{aligned}
   2\dt\, \bb(\psi^n, \omega^n, \omega^{n+1})
	&= 2\dt\, \bb(\psi^n,  \omega^{n+1}, \omega^{n+1}-\omega^{n})\\
	&\le 2C_w \dt \|\nabla^\perp\psi^n\|_{H^1}  \|\omega^{n+1}\|_{H^1}\| \omega^n-\omega^{n+1}\|_2\\
	&\le \sfrac{1}{2} \|\omega^{n+1}-\omega^{n}\|^2_2+2C_w^2 \dt^2 \|\nabla^\perp\psi^n\|^2_{H^1}  \|\omega^{n+1}\|_{H^1}^2\\
	&\le \sfrac12 \|\omega^{n+1}-\omega^{n}\|_2^2+2C_w^2 \dt^2 \|\omega^n\|^2_2 \|\omega^{n+1}\|_{H^1}^2.
\end{aligned}\end{equation}
Relations (\ref{1})--(\ref{3}) imply
\begin{equation}\label{4}\begin{aligned}
     \|\om^{n+1}\|_2^2 -\|\om^{n}\|_2^2 + \sfrac{1}{2}  \|\omega^{n+1} - \omega ^{n}\|_2^2
         + (\nu -2C_w^2 \dt \|\omega^n\|_2^2) \dt\,&\| \omega ^{n+1}\|_{H^1}^2\\
	&\leq    \frac{c_0^2}{\nu}\dt \|f^n\|_2^2.
\end{aligned}\end{equation}
Here and in what follows, $C$ and $c$ denote generic constants whose value
may not be the same each time they appear.
Numbered constants, e.g., $c_{42}^{}$, have fixed values.

We are now able to prove the following:

\medskip
\begin{lemma}\label{t:bdh}
 Let $\omega_0 \in \Lr^2$ and let $\omega^n$ be the solution of the numerical scheme $($\ref{scheme}$)$. Also, let  $f \in L^\infty(\Real_+;H)$ and set $\|f\|_{\infty} := \|f\|_{L^\infty(\Real_+;H)}$.
  Then there exists $M_0 = M_0(\|\omega_0\|_2, \nu, \|f\|_{\infty})$ such that if
\begin{equation}\label{5a}
  \dt \leq \frac{\nu}{4C_w^2M_0^2},
\end{equation}
then
\begin{equation}\label{q:bdinh}
   \|\omega^{n}\|_2 \leq M_0, \quad \forall \, n \geq 0,
\end{equation}
\begin{equation}\label{q:bdv}
   \|\omega^{n}\|_2^2 \le \left(1+  \frac{\nu}{2c_0^2} \dt \right)^{-n} \|\omega_0\|_2^2
           +  \frac{2c_0^4}{ \nu^2} \fvi^2 \left[ 1 - \left( 1 + \frac{\nu}{2c_0^2} \dt \right)^{-n} \right],
    \>\forall\, n\ge0,
\end{equation}
and
 \begin{eqnarray}\label{6}
     \frac{ \nu}{2} \dt \sum_{n=i}^{m} \|\omega^{n}\|_{H^1}^2
     \leq \|\omega^{i-1}\|_2^2 + \frac{c_0^2}{\nu} \fvi^2 (m-i+1) \dt, \,
             \quad \forall \, i=1, \cdots, m.
\end{eqnarray}\end{lemma}

\begin{proof}
We will first prove (\ref{q:bdv}) by induction on $n$.
It is clear that (\ref{q:bdv}) holds for  $n=0$.
Assuming that (\ref{q:bdv})
holds for $n=0, \cdots, m$, we then have that  (\ref{q:bdinh}) holds
for $n=0, \cdots, m$, where
\begin{equation}
   M_0^2 = M_0^2(\|\omega_0\|_2,\nu,\fvi)
	= \|\omega_0\|_2^2+\frac{2c_0^4}{ \nu^2} \fvi^2.
\end{equation}
Then (\ref{4}) and (\ref{5a}) yield
\begin{equation}\label{7}
     \|\om^{n+1}\|_2^2 -\|\om^{n}\|_2^2 +\frac{1}{2}  \|\omega^{n+1} - \omega ^{n}\|_2^2
         +\frac{\nu}{2} \dt \| \omega ^{n+1}\|^2_{H^1} \leq \frac{c_0^2}{\nu}\dt \|f^n\|_2^2
\end{equation}
for all $n=0, \cdots, m$.
Using again the Poincar\'e inequality, the above inequality implies
\begin{equation}\label{8}
       \|\om^{n+1}\|_2^2  \leq \frac{1}{\alpha} \|\om^{n}\|_2^2
         + \frac{c_0^2}{\alpha \nu } \dt \|f^n\|_2^2,
\end{equation}
where
\begin{equation}\label{9}
      \alpha = 1+ \frac{\nu}{2c_0^2} \dt.
\end{equation}
Using recursively (\ref{8}), we find
\begin{equation}\label{10}\begin{aligned}
   \|\om^{m+1}\|_2^2
	&\leq \frac{1}{\alpha^{m+1}}\|\omega^0\|_2^2
                        + \frac{c_0^2}{\nu} \dt  \sum_{i=1}^{m+1}\frac{1}{\alpha^i}\|f^{m+1-i}\|_2^2 \\
                 &\leq \biggl(1+  \frac{\nu}{2c_0^2} \dt \biggr)^{-m-1}\|\omega_0\|_2^2 + \frac{2c_0^4}{ \nu^2} \fvi^2
                         \biggl[1- \biggl(1+\frac{\nu}{2c_0^2} \dt\biggr)^{-m-1} \biggr],
\end{aligned}\end{equation}
and thus (\ref{q:bdv}) holds for $n=m+1$.
We therefore have that (\ref{q:bdv}) holds for $n \geq 0$ and (\ref{q:bdinh}) follows right away.

Now adding inequalities (\ref{7}) with $n$ from $i$ to $m$ and
dropping some positive terms, we find
\begin{equation}\label{11}\begin{aligned}
     \frac{ \nu}{2} \dt \sum_{n=i}^{m} \|\omega^{n+1}\|^2_{H^1}
     &\leq \|\omega^{i}\|_2^2 +
             \frac{c_0^2}{ \nu }\dt \sum_{n=i}^{m}\|f^n\|_2^2 \\
      &\leq  \|\omega^{i}\|_2^2 + \frac{c_0^2}{\nu} \fvi^2 (m-i+1) \dt,
\end{aligned}\end{equation}
which is exactly (\ref{6}). This completes the proof of Lemma \ref{t:bdh}.
\end{proof}

\medskip
\begin{cor}\label{C1}
If
\begin{equation}\label{q:k0}
   0 < \dt \leq \min\left\{\frac{\nu}{4C_w^2M_0^2}, \frac{2c_0^2}{\nu}
   \right\}=:k_0,
\end{equation}
then
\begin{equation}
\label{q:tabs}
    \|\om^{n}\|_2^2\leq 2 \rho_0^2,
        \quad \forall \, n\dt  \geq T_0(\|\omega_0\|_2,\fvi)
           :=\frac{8 c_0^2}{\nu } \ln\left(\frac{\|\omega_0\|_2}{\rho_0}\right),
\end{equation}
where $\rho_0:=(\sqrt{2}c_0^2/\nu)\fvi$.
\end{cor}

\smallskip
\begin{proof}
From the bound \eqref{q:bdv} on $\|\om^{n}\|_2^2$, we infer that
\[
  \|\om^{n}\|_2^2 \leq \Bigl( 1 + \frac{\nu}{2c_0^2} \dt  \Bigr)^{-n} \|\omega_0\|_2^2 + \rho_0^2,
\]
and using assumption \eqref{q:k0} on $\dt$ and the fact that $1+x \geq
\exp(x/2)$ if $x\in(0,1)$, we obtain
\[
  \|\om^{n}\|_2^2 \leq \exp\Bigl(-n \dt \frac{\nu }{4c_0^2}\Bigr)\|\omega_0\|_2^2 + \rho_0^2.
\]
For $n\dt \geq T_0$, the last inequality implies the conclusion
\eqref{q:tabs} of the Corollary.
\end{proof}

\medskip
Now we show that the $H^1$ norm is also bounded uniformly in time under the same kind of constraint as for the $L^2$ estimate.
To this end, we first prove that $\omega^n$ is bounded for $n
\leq N$, for some $N$, and then, with the aid of a version of the
discrete uniform Gronwall lemma, we show that $\omega^n$ is bounded
for all $n \geq N$.

More precisely, we have the following:

\medskip
\begin{lemma}\label{finite}
Let $\omega_0 \in \Lr^2$ and let $\omega^n$ be the solution of the numerical scheme (\ref{scheme}).
Also, let $\dt\leq k_0$, with $k_0$ as in Corollary \ref{C1}, and let $r\ge 8 c_0^2/\nu$ be arbitrarily fixed.
Then, for $n=1,\cdots, N_0+N_r-1$,
\begin{eqnarray}\label{12}
      \| \omega^{n}\|^2_{H^1}
      \leq 4^{({2C_w^2}/{\nu}) M_0^2 (T_0+r)}\Bigl(\| \omega_0\|^2_{H^1}+ \frac{1}{C_w^2 M_0^2} \fvi^2\Bigr)
\end{eqnarray}
where $N_0= \lfloor T_0/\dt \rfloor$, with $N_r= \lfloor r/\dt \rfloor$
and $T_0$ that in Corollary~\ref{C1}.
\end{lemma}

\medskip
\begin{proof}
Taking the scalar product of (\ref{scheme}) with $-2\dt\, \Delta \omega^{n+1}$,
we obtain
\begin{equation}\label{13}\begin{aligned}
    \|\omega^{n+1}\|^2_{H^1} - \|\omega^{n}\|^2_{H^1}
     &+ \|\omega^{n+1} - \omega^{n}\|^2_{H^1} + 2 \nu \dt \|\Delta \omega^{n+1}\|^2_2\\
     &- 2  \dt\, \bb(\psi^n, \omega^n,\Delta  \omega^{n+1})
	= -2\dt(f^n,\Delta  \omega^{n+1})_{L^2}.
\end{aligned}\end{equation}
We bound the right-hand side of~(\ref{13}) using the Cauchy--Schwarz inequality,
\begin{equation}\label{14}
 -2 \dt (f^n,\Delta  \omega^{n+1})_{L^2}
    \leq 2\dt \|f^n\|_2 \|\Delta  \omega^{n+1}\|_2 \leq \frac{\nu}{2} \dt \|\Delta  \omega^{n+1}\|_2^2 + \frac{2}{\nu} \dt \|f ^n\|_2^2.
\end{equation}
Using
the Wente type estimate (\ref{q:wente2}), we bound the nonlinear term as
\begin{align}
2\dt\, \bb(\psi^n, \omega^n,&\Delta  \omega^{n+1})
	= 2\dt\,\bb(\psi^n,   \omega^n-\omega^{n+1}, \Delta  \omega^{n+1})\notag\\
	&\hbox to80pt{} {}+ 2\dt\, \bb(\psi^n,   \omega^{n+1}, \Delta  \omega^{n+1})\notag\\
	&\le 2C_w \dt \|\nabla^\perp\psi^n\|_{H^1} \| \omega^{n+1}-\omega^{n}\|_{H^1} \|\Delta  \omega^{n+1}\|_2 \notag\\
	&\qquad {}+ 2C_w \dt \|\nabla^\perp\psi^n\|_{H^1} \| \omega^{n+1}\|_{H^1} \|\Delta  \omega^{n+1}\|_2 \notag\\
	&\leq \frac{1}{2} \|\omega^{n+1}-\omega^{n}\|^2_{H^1}+2C_w^2 \dt^2 \|\nabla^\perp\psi^n\|^2_{H^1}  \|\Delta  \omega^{n+1}\|_2^2\notag\\
	&\qquad {}+\frac{\nu}{2} \dt \|\Delta  \omega^{n+1}\|_2^2 + \frac{2C_w^2}{\nu} \dt
\|\nabla^\perp\psi^n\|^2_{H^1} \| \omega^{n+1}\|^2_{H^1}\notag\\
	&\leq \frac{1}{2} \|\omega^{n+1}-\omega^{n}\|^2_{H^1}+2C_w^2 \dt^2 \|\omega^n\|_2^2  \|\Delta  \omega^{n+1}\|_2^2\notag\\
	&\qquad {}+\frac{\nu}{2} \dt \|\Delta  \omega^{n+1}\|_2^2 + \frac{2C_w^2}{\nu} \dt
\|\omega^n\|_2^2 \| \omega^{n+1}\|^2_{H^1}.\label{15}
\end{align}
Relations (\ref{13})--(\ref{15}) imply
\begin{equation}\label{16}\begin{aligned}
   \Bigl(1-\frac{2C_w^2}{\nu} \| \omega^{n}\|_2^2 \dt \Bigr)\|\omega^{n+1}\|^2_{H^1} &- \|\omega^{n}\|^2_{H^1}
      + \frac{1}{2}\|\omega^{n+1} - \omega^{n}\|^2_{H^1}\\
	&+ \left( \nu -2C_w^2 \dt M_0^2\right) \dt \|\Delta \omega^{n+1}\|_2^2
	\leq \frac{2}{\nu} \dt \|f^n\|_2^2,
\end{aligned}\end{equation}
from which we find
\begin{equation}\label{17}
      \| \omega^{n+1}\|^2_{H^1} \leq \frac{1}{\alpha}\|\omega^{n}\|^2_{H^1}
         + \frac{2}{\alpha \nu} \dt \fvi^2,
\end{equation}
where
\begin{equation}\label{18}
      \alpha = 1-\frac{2C_w^2}{\nu} \dt M_0^2 >0.
\end{equation}
Using recursively (\ref{17}), we find
\begin{eqnarray}\label{19}
      \| \omega^{n+1}\|^2_{H^1} &\leq& \frac{1}{ \alpha^{n+1}}\| \omega_0\|^2_{H^1}+ \frac{2}{\nu} \dt \fvi^2 \sum_{i=1}^{n+1}\frac{1}{
      \alpha^i}\nonumber \\
      &\leq& \left(1-\frac{2C_w^2}{\nu} \dt M_0^2\right)^{-1-n}\left[\| \omega_0\|^2_{H^1}+ \frac{1}{C_w^2 M_0^2} \fvi^2\right].
\end{eqnarray}
Since $2C_w^2  M_0^2 \dt/\nu \leq 1/2$ by hypothesis \eqref{q:k0} and
\[
   1-x \geq 4^{-x} \textrm{ if } x \in (0,1/2),
\]
relation \eqref{19} gives conclusion \eqref{12} of Lemma
\ref{finite}. Thus, the lemma is proved.
\end{proof}

In order to obtain a uniform bound valid for $n \geq N_0+N_r$, we need the
following discrete uniform Gronwall lemma, which has been proved in
\cite{T2009} and we repeat here for convenience.

\medskip
\begin{lemma}\label{dugronwall}
We are given $\dt>0$, positive integers $ n_0, n_1$, and positive
sequences $\xi_n$, $\eta_n$, $\zeta_n$ such that
\begin{align}
   &\dt  \eta_{n+1} < \sfrac{1}{2}, \quad\forall n \geq n_0,\label{time}\\
   &(1 - \dt \eta_{n+1}) \xi_{n+1} \le \xi_{n} + \dt \zeta_{n+1},
          \quad\forall n \geq n_0.\label{gronseq2}
\end{align}
Assume also that
\begin{equation}\label{groncond}\begin{aligned}
  &\dt \sum_{n=n_2}^{n_2+n_1+1} \eta_n\le a_1, \\
  &\dt \sum_{n=n_2}^{n_2+n_1+1}  \zeta_n \le a_2,\\
  &\dt \sum_{n=n_2}^{n_2+n_1+1}  \xi_n \le a_3,
\end{aligned}\end{equation}
for all  $n_2 \ge n_0$. We then have,
\begin{equation}\label{ugronest}
       \xi_{n+1} \le \Bigl( \frac{a_3}{\dt n_1}
         + a_2 \Bigr)\, {\rm e}^{4a_1}, \quad\forall n > n_0+n_1.
\end{equation}
\end{lemma}

\begin{proof}
Let $m_1$ and $m_2$ be such that $n_0 < m_1 \leq m_2 \leq m_1+n_1$.
Using recursively \eqref{gronseq2}, we derive
\begin{equation}\begin{aligned}
  \xi_{m_1+n_1+1} &\le \prod_{n=m_2}^{m_1+n_1+1}\frac{1}{1-\dt \eta_n} \xi_{m_2-1}
    + \dt \sum_{n=m_2}^{m_1+n_1+1} \zeta_n
        \prod_{j=n}^{m_1+n_1+1} \frac{1}{1-\dt \eta_j}.
\end{aligned}
\end{equation}
Using the fact that $1-x \ge \textrm{e}^{-4x}, \,\forall x \in
\left(0, \frac{1}{2} \right)$, and recalling assumptions
${\eqref{time}}$, and the first and second conditions in
${\eqref{groncond}}$, we obtain
\[
  \xi_{m_1+n_1+1} \le (\xi_{m_2-1} + a_2) e^{4a_1}.
\]
Multiplying this inequality by $\dt$, summing $m_2$ from $m_1$ to
$m_1+n_1$ and using the third assumption ${\eqref{groncond}}$ gives
conclusion \eqref{ugronest} of the lemma.
\end{proof}

\medskip
We are now able to derive a uniform bound for $\|\omega^n\|_{H^1}$ valid
for sufficiently large $n$. More precisely, we have the following:

\medskip
\begin{lemma}\label{5}
Let $\omega_0 \in \Lr^2$ and let $\omega^n$ be the solution of the numerical scheme
(\ref{scheme}).
Also, let $\dt\leq k_0$, with $k_0$ as in Corollary~\ref{C1}.
Then  there exist constants $M_1=M_1(\nu, \fvi)$ and
$N=N(\|\omega_0\|_2, \nu, \fvi)$ such that
\begin{equation}\label{20}
     \| \omega^n\|_{H^1} \le M_1,  \quad \forall n\ge N.
\end{equation}
\end{lemma}

\begin{proof}
Let $\dt$ be as in the hypothesis,  $T_0$ be as in Corollary
\ref{C1}, $r$ as in Lemma \ref{finite} and set $N_0:=\lfloor T_0/\dt \rfloor$. We will apply Lemma
\ref{dugronwall} to \eqref{16}, with $\xi_n=\| \omega^{n}\|^2_{H^1}$,
$\eta_n=2C_w^2 \| \omega^{n-1}\|_2^2/\nu$, $\zeta_n=2\fvi^2/ \nu$,
$n_0=N_0+2$, $n_1=N_r-2$. For $n_2 \ge n_0$, we compute (taking into
account that, by \eqref{q:tabs}, $\| \omega^{n}\|_2^2 \leq 2\rho_0^2$,
for $n\geq N_0$):
\begin{eqnarray}
  &\dt& \sum_{n=n_2}^{n_2+n_1+1} \eta_n =  \dt \sum_{n=n_2}^{n_2+n_1+1} \frac{2C_w^2 }{\nu}\| \omega^{n-1}\|_2^2\le \frac{4C_w^2 }{\nu}\rho_0^2r:=a_1, \\
  &\dt& \sum_{n=n_2}^{n_2+n_1+1}  \zeta_n = \dt \sum_{n=n_2}^{n_2+n_1+1} \frac{2}{\nu}\fvi^2\leq\frac{2}{\nu}\fvi^2r:=a_2,\\
  &\dt& \sum_{n=n_2}^{n_2+n_1+1}  \xi_n=\dt \sum_{n=n_2}^{n_2+n_1+1}\| \omega^{n}\|^2_{H^1} \quad (\text{by } \eqref{6})\\
  && \quad \quad \quad \quad \quad \leq \frac{2}{\nu}\left(\|\omega^{n_2-1}\|_2^2 + \frac{c_0^2}{\nu} \fvi^2 (n_1+2) \dt\right) \quad (\text{by } \eqref{q:tabs})\\
  && \quad \quad \quad \quad \quad \leq \frac{2}{\nu}\left[2\rho_0^2 + \frac{c_0^2}{\nu} \fvi^2 r\right]=:a_3.
\end{eqnarray}
By \eqref{ugronest}, we obtain
\begin{align}
   \|\omega^{n}\|^2_{H^1}
	&\le \left[ \frac{4}{\nu}\left(\frac{2\rho_0^2}{r} + \frac{1}{\nu\lb} \fvi^2 \right)+\frac{2}{\nu}\fvi^2r\right] \exp\left( \frac{16C_w^2 }{\nu}\rho_0^2r\right)\\
	&=: M_1^2(\nu, \fvi), \quad \forall n \geq N_0+N_r.
\end{align}
Taking $N=N_0+N_r$, we obtain conclusion \eqref{20} of Lemma \ref{5}.
\end{proof}

\medskip
We can summarize the above results in the following:

\medskip
\begin{theorem}\label{stability}
   The classical scheme (\ref{scheme}) defines a discrete dynamical system on $\Lr^2$ that is long time stable in both $L^2$ and $H^1$ norms.
More precisely, for any $\omega_0 \in \Lr^2$,
there exist  constants $k_0 = k_0(\|\omega_0\|_2, \nu, \fvi)$,
$    M_0 = M_0(\|\omega_0\|_2, \nu, \fvi)$,
$    M_1=M_1(\nu, \fvi)$ and
$    N=N(\|\omega_0\|_2, \nu, \fvi)$
such that
\begin{eqnarray}
     \|\omega^n\|_2 &\le& M_0, \quad \forall n \ge 0, \forall k\in (0, k_0),
     \\
     \| \omega^n\|_{H^1} &\le& M_1,  \quad \forall n\ge N, \forall k \in (0, k_0).
\end{eqnarray}
\end{theorem}

\subsection{Convergence of long time statistics}

Here we show that, with time-independent forcing,
the long time statistical properties as well as the global attractors of the scheme (\ref{scheme}) converge to that of the Navier--Stokes sytem (\ref{NSE}) at vanishing time step size.
This is a straightforward application of the abstract convergence result (Prop.~\ref{abs_conv_stat}) in Appendix~\ref{s:conv}, which itself is a slight modification of the results presented in \cite{W2010}.

\medskip
\begin{theorem}\label{conv_stat}
Let $\dy_tf=0$.
The global attractor and the long time statistical properties of the classical scheme (\ref{scheme}) converge to that of the Navier--Stokes system (\ref{NSE}) at vanishing time step.
\end{theorem}

\medskip
\begin{proof}
We use the abstract convergence result Prop.~\ref{abs_conv_stat},
taking $X= B(0,{\|f\|_2^{}}/{\nu})$, i.e.\ a ball in $\Lr^2$
centered at the origin with radius ${\|f\|_2^{}}/{\nu}$.
(The size of the ball needs to be adjusted depending on the absorbing property of the scheme.)

The uniform continuity (H5) of the Navier--Stokes system (\ref{NSE}) is a classical result \cite{CF1988, T1983}.
The uniform dissipativity (H3) of the scheme (\ref{scheme}) for small enough time step with the choice of the phase space $X$ follows from Theorem \ref{stability}.
The uniform convergence on finite time interval (H4) is proved
in Lemma~\ref{t:error} below.
\end{proof}

\medskip
\begin{lemma}\label{t:error}
Let $\omega$ be the solution of the continuous system \eqref{NSE} with
$\omega(0)=\omegaz\in\Attr$ and $\omega^n$ that of \eqref{scheme}
with $\omega^0=\omegaz$.
Assume that $f$ is sufficiently smooth so that
\begin{equation}
   M_V := \sup_{\omega\in\Attr}\,\bigl(\|\dy_{tt}\omega\|_{H^{-1}}^2 + \|\omega\|_{L^2}^2\|\dy_t\omega\|_{L^2}^2\bigr) < \infty,
\end{equation}
and that Theorem~\ref{stability} holds.
Then for $\dt<k_0$ one has
\begin{equation}
   \|\omega^n-\omega(n\dt)\|_2^2 \le \dt\,C(M_0,M_V;\nu)
\end{equation}
for all $0\le n\dt\le 1$.
\end{lemma}

\medskip
\begin{proof}
We follow the approach in \cite[\S17]{MT1998} and take $\dy_tf=0$.
For notational convenience, we write $t_n:=n\dt$ and $\omega_n:=\omega(n\dt)$.
Using the identity
\begin{equation}
   \int_{n\dt}^{(n+1)\dt} (t-n\dt)\,\dy_{tt}\omega(t) \dtt
	= \dt\,\dy_t\omega\big|_{(n+1)\dt}^{} - \omega_{n+1} + \omega_n\,,
\end{equation}
we have
\begin{equation}
   \frac{\omega_{n+1}-\omega_n}{\dt} + \sgb\psi_n\cdot\gb\omega_n
	- \nu\Delta\omega_{n+1} = f + R_{n+1}\,.
\end{equation}
Here $-\Delta\psi_n:=\omega_n$ and the local truncation error is
\begin{equation}\label{q:rndef}
   -R_{n+1} := \sgb\delta\psi_{n+1}\cdot\gb\omega_{n}
	- \sgb\psi_{n+1}\cdot\gb\delta\omega_{n+1}
	+ \frac{1}{\dt} \int_{n\dt}^{(n+1)\dt} \!(t-n\dt)\, \dy_{tt}\omega(t) \;\mathrm{d}t
\end{equation}
with
\begin{equation}
   \delta\omega_{n+1}:=\omega_{n+1}-\omega_n
	= \int_{n\dt}^{(n+1)\dt} \dy_t\omega(t) \dtt
   \quad\textrm{and}\quad
   -\Delta\delta\psi_{n+1} := \delta\omega_{n+1}.
\end{equation}

We now consider the error $e^n:=\omega_n-\omega^n$, which satisfies
\begin{equation}\begin{aligned}
   \frac{e^{n+1}-e^n}{\dt} - \nu\Delta e^{n+1}
	&= \sgb\psi^n\cdot\gb\omega^n - \sgb\psi_n\cdot\gb\omega_n + R_{n+1}\\
	&= -\sgb\psi_n\cdot\gb e^n - \sgb\phi^n\cdot\gb\omega^n + R_{n+1}\\
\end{aligned}\end{equation}
with $e^0=0$ and $-\Delta\phi^n:=e^n$.
Multiplying by $2\dt\, e^{n+1}$, we find
\begin{equation}\begin{aligned}
   \|e^{n+1}\|_2^2 &- \|e^n\|_2^2 + \|e^{n+1}-e^n\|_2^2
	+ 2\nu\dt\|e^{n+1}\|_{H^1}^2\\
	&+ 2\dt\,\bb(\psi_n,e^{n+1},e^{n+1}-e^n)
	+ 2\dt\,\bb(\phi^n,\omega^n,e^{n+1})\\
	&= 2\dt\,(R_{n+1},e^{n+1}).
\end{aligned}\end{equation}
Bounding the nonlinear terms as
\begin{equation}\begin{aligned}
   2\dt\,(\sgb\psi_n\cdot\gb e^{n+1},e^{n+1}-e^n)
	&\le \|e^{n+1}-e^n\|_2^2 + \dt^2\|\sgb\psi_n\cdot\gb e^{n+1}\|_2^2\\
	&\le \|e^{n+1}-e^n\|_2^2 + C_w^2\dt^2\|\omega_n\|_2^2\,\|\gb e^{n+1}\|_2^2
\end{aligned}\end{equation}
where \eqref{q:wente2} has been used for the second inequality, and
\begin{equation}\begin{aligned}
   2\dt\,(\sgb\phi^n\cdot\gb\omega^n,e^{n+1})
	&\le 2\dt\,\|\sgb\phi^n\cdot\gb e^{n+1}\|_2^{}\|\omega^n\|_2^{}\\
	&\le 2 C_w\dt\,\|e^n\|_2^{}\|e^{n+1}\|_{H^1}^{}\|\omega^n\|_2^{}\\
	&\le \nu\dt\,\|e^{n+1}\|_{H^1}^2 + \frac{C_w^2\dt}{\nu}\,\|\omega^n\|_2^2\,\|e^n\|_2^2\,,
\end{aligned}\end{equation}
we obtain,
noting that $\dt\le k_0$ implies $\nu-C_w^2\dt\,\|\omega_n\|_2^2\ge\nu/2>0$,
\begin{equation}\begin{aligned}
   \|e^{n+1}\|_2^2 + \dt\,\bigl(\nu &- C_{w}^2\dt\,\|\omega_n\|_2^2\bigr)\,\|e^{n+1}\|_{H^1}^2\\
	&\le \Bigl(1 + \frac{C_w^2\dt}{\nu}\|\omega^n\|_2^2\Bigr)\,\|e^n\|_2^2
	+ c\dt\,\|R_{n+1}\|_{H^{-1}}^2\,.
\end{aligned}\end{equation}

It remains to bound $R_{n+1}$ in $H^{-1}$, so for the second term in
\eqref{q:rndef} we compute, for any fixed $\vfi\in \Hr^1$,
\begin{equation}\begin{aligned}
   \bigl|b\bigl(\psi_{n+1},\dy_t\omega,\vfi\bigr)\bigr|
	&= \bigl|\bigl(\sgb\vfi\cdot\gb\psi_{n+1},\dy_t\omega\bigr)_{L^2}^{}\,\bigr|\\
	&\le C_w\,\|\vfi\|_{H^1}^{}\|\psi_{n+1}\|_{H^2}^{}\|\dy_t\omega\|_{L^2}^{}
\end{aligned}\end{equation}
where \eqref{q:wente2} and the identity $b(p,q,r)=b(q,r,p)=b(r,p,q)$
have been used.
Similarly, for the first term,
\begin{equation}\begin{aligned}
   \bigl|b\bigl(\omega_n,\dy_t\psi,\vfi\bigr)\bigr|
	&= \bigl|\bigl(\sgb\vfi\cdot\gb\dy_t\psi,\omega_n\bigr)_{L^2}^{}\,\bigr|\\
	&\le C_w\,\|\vfi\|_{H^1}^{}\|\omega_n\|_{L^2}^{}\|\dy_t\psi\|_{H^2}^{}.
\end{aligned}\end{equation}
The last term in \eqref{q:rndef} is readily bounded, and we have by Cauchy--Schwarz,
\begin{equation}\begin{aligned}
   \|R_{n+1}\|_{H^{-1}}^2
	&\le c\,\dt\!\sup_{t\in[n\dt,(n+1)\dt]}\|\omega(t)\|_{L^2}^2
		\int_{n\dt}^{(n+1)\dt} \|\dy_t\omega(t)\|_{L^2}^2 \dtt\\
	&\hbox to120pt{} {}+ {\dt} \int_{n\dt}^{(n+1)\dt} \|\dy_{tt}\omega(t)\|_{H^{-1}}^2 \dtt.
\end{aligned}\end{equation}

The following bound then follows easily
\begin{align}
   \|\omega_{n+1}-\omega^{n+1}\|_2^2 &= \|e^{n+1}\|_2^2
	\le c\,\Bigl(1 + \frac{c\dt}{\nu}M_0^2\Bigr)^{n+1}
	\sum_{j=0}^n\, \dt\,\|R_{j+1}\|_{H^{-1}}^2\notag\\
	&\le c\,\dt^2\,\exp\Bigl(\frac{c\, (n+1)\dt}{\nu}\,M_0^2\Bigr) M_2((n+1)\dt)
\end{align}
where
\begin{equation}
   M_2(t) := \int_0^{t} \|\dy_{tt}\omega(t')\|_{H^{-1}}^2 \dtt'
	+ \sup_{t'\in[0,t]}\|\omega(t')\|_{L^2}^2
		\int_0^{t} \|\dy_t\omega(t')\|_{L^2}^2 \dtt',
\end{equation}
and with it the lemma.
\end{proof}

\section{Galerkin Fourier spectral approximation}

This section is devoted to the long time stability of the following Galerkin Fourier spectral approximation of the two dimensional Navier--Stokes equations
\begin{equation}\label{Galerkin_scheme}
    \frac{\omega_N^{n+1}-\omega_N^n}{\dt} + P_N(\nabla^\perp\psi_N^n\cdot\nabla\omega_N^n) - \nu\Delta\omega_N^{n+1} = P_N(f^n).
\end{equation}
where $\omega_N^n$, $\psi_N^n  \in \mathcal{P}_N := \{\textrm{all trigonometric functions on $\Omega$ with frequency in each}$
$\textrm{direction at most $N$}\}$.
$P_N$ is defined as the orthogonal projection from $\Lr^2(\Omega)$ onto $\mathcal{P}_N$.

Just like for the semi-discrete scheme \eqref{scheme}, we can  show that  the scheme (\ref{Galerkin_scheme})  is uniformly bounded
in $L^2$, provided that the time step is sufficiently small. More precisely, we have the following:

\medskip
\begin{lemma}\label{l:bdh}
 Let $\omega_0 \in \Lr^2$ and let $\omega_N^n$ be the solution of the numerical scheme (\ref{Galerkin_scheme}). Also, let  $f \in L^\infty(\Real_+;H)$ and set $\|f\|_{\infty} := \|f\|_{L^\infty(\Real_+;H)}$.
  Then there exists $M_0 = M_0(\|\omega_0\|_2, \nu, \|f\|_{\infty})$ such that if
  \begin{equation} \label{46}
  \dt \leq \frac{\nu}{4C_w^2M_0^2},
  \end{equation}
   then
\begin{equation}\label{47}
   \|\omega^n_N\|_2 \leq M_0, \, \forall \, n \geq 0,
\end{equation}
\begin{equation}\label{48}
   \|\omega^n_N\|_2^2 \le \left(1+  \frac{\nu }{2c_0^2} \dt \right)^{-n} \|\omega_0\|_2^2
           +  \frac{2c_0^4}{ \nu^2} \fvi^2 \left[ 1 - \left( 1 + \frac{\nu}{2c_0^2} \dt \right)^{-n} \right]
                ,
    \>\forall\, n\ge0,
\end{equation}
and
 \begin{eqnarray}\label{49}
     \frac{ \nu}{2} \dt \sum_{n=i}^{m} \|\omega_N^{n}\|^2_{H^1}
     \leq \|\omega_N^{i-1}\|_2^2 + \frac{c_0^2}{\nu} \fvi^2 (m-i+1) \dt, \,
             \quad \forall \, i=1, \cdots, m.
\end{eqnarray}
\end{lemma}

\medskip
\begin{proof}
Taking the scalar product of (\ref{Galerkin_scheme}) with $2\dt\, \omega_N^{n+1}$ we obtain
\begin{equation}\label{50}\begin{aligned}
     \|\om_N^{n+1}\|_2^2 -\|\om_N^{n}\|_2^2 &+ \|\omega_N^{n+1} - \omega_N^{n}\|_2^2
         + 2 \nu \dt \| \omega_N ^{n+1}\|^2_{H^1}\\
	&+ 2\dt\, \bb(\psi_N^n, \omega_N^n, \omega_N^{n+1})
	= 2  \dt (f^n, \omega_N ^{n+1})_{L^2}.
\end{aligned}\end{equation}
Using the Cauchy--Schwarz inequality and the Poincar\'e inequality, we have the following bound for the right-hand side of (\ref{50}):
\begin{equation}\label{51}\begin{aligned}
      2  \dt (f^n, \omega_N ^{n+1})_{L^2} \leq 2\dt \|f^n\|_2 \|\omega_N^{n+1}\|_2
       &\leq 2\dt c_0 \|f^n\|_2 \|\omega _N^{n+1}\|_{H^1}\\
       &\leq \nu \dt \|\omega _N^{n+1}\|^2_{H^1} + \frac{c_0^2}{\nu}\dt \|f^n\|_2^2,
\end{aligned}\end{equation}
whereas the nonlinear term can be bounded using the Wente type inequality \eqref{q:wente2} as
\begin{equation}\label{52}\begin{aligned}
   2\dt\, \bb(\psi_N^n, \omega_N^n,\, &\omega_N^{n+1})
	= 2\dt\, \bb(\psi_N^n,  \omega_N^{n+1}, \omega_N^{n+1}-\omega_N^{n})\\
	&\leq 2C_w \dt \|\nabla^\perp\psi_N^n\|_{H^1}  \|\omega_N^{n+1}\|_{H^1} \| \omega_N^n-\omega_N^{n+1}\|_2 \\
	&\leq \frac{1}{2} \|\omega_N^{n+1}-\omega_N^{n}\|_2^2+2C_w^2 \dt^2
\|\nabla^\perp\psi_N^n\|^2_{H^1}  \|\omega_N^{n+1}\|^2_{H^1}\\
	&\leq \frac{1}{2}
\|\omega_N^{n+1}-\omega_N^{n}\|_2^2+2C_w^2 \dt^2 \|\omega_N^n\|_2^2
\|\omega_N^{n+1}\|^2_{H^1}.
\end{aligned}\end{equation}
Relations (\ref{50})--(\ref{52}) imply
\begin{equation}\label{53}\begin{aligned}
     \|\om_N^{n+1}\|_2^2 -\|\om_N^{n}\|_2^2 +\frac{1}{2}  \|\omega_N^{n+1} - \omega _N^{n}\|_2^2
         &+ (\nu -2C_w^2 \dt \|\omega_N^n\|_2^2) \dt \| \omega_N^{n+1}\|^2_{H^1}\\
	&\leq    \frac{c_0^2}{\nu}\dt \|f^n\|_2^2.
\end{aligned}\end{equation}
By induction, one can prove that if $\dt$ satisfies \eqref{46}, then
\begin{align}
   \|\omega^n_N\|_2^2
	&\le \left(1+  \frac{\nu}{2c_0^2} \dt \right)^{-n} \|\omega_N^0\|_2^2
           +  \frac{2c_0^4}{ \nu^2} \fvi^2 \left[ 1 - \left( 1 + \frac{\nu}{2c_0^2} \dt \right)^{-n} \right] \notag\\
	&\le \left(1+  \frac{\nu}{2c_0^2} \dt \right)^{-n} \|\omega_0\|_2^2
           +  \frac{2c_0^4}{ \nu^2} \fvi^2 \left[ 1 - \left( 1 + \frac{\nu}{2c_0^2} \dt \right)^{-n} \right] \notag\\
	&\le  \|\omega_0\|_2^2 +  \frac{2c_0^4}{ \nu^2} \fvi^2
	=:M_0^2(\|\omega_0\|_2, \nu, \|f\|_{\infty}),
    \>\forall\, n\ge0,\label{54}
\end{align}
from which conclusions \eqref{47} and \eqref{48} of the Lemma follow right away.

Adding inequalities (\ref{53}) with $n$ from $i$ to $m$ and recalling the bound \eqref{47} and the time restriction \eqref{46}, we find
\begin{eqnarray}\label{11b}
     \frac{ \nu}{2} \dt \sum_{n=i}^{m} \|\omega_N^{n+1}\|^2_{H^1}
     & \leq &\|\omega_N^{i}\|_2^2 +
             \frac{c_0^2}{ \nu }\dt \sum_{n=i}^{m}\|f^n\|_2^2 \\
      &\leq& \|\omega_N^{i}\|_2^2 + \frac{c_0^2}{\nu} \fvi^2 (m-i+1) \dt,
\end{eqnarray}
which is exactly conclusion (\ref{49})  of Lemma \ref{l:bdh}. This completes the proof of the lemma.
\end{proof}

From bound \eqref{48} we can also derive the following
  \begin{cor}\label{C2}
If
\begin{equation}\label{k0}
   0 < \dt \leq \min\left\{\frac{\nu}{4C_w^2M_0^2}, \frac{2c_0^2}{\nu}
   \right\}=:k_0,
\end{equation}
then
\begin{equation}\label{58}
    \|\om_N^{n}\|_2^2\leq 2 \rho_0^2,
        \quad \forall \, n\dt  \geq T_0(\|\omega_0\|_2,\fvi)
           :=\frac{8 c_0^2}{\nu } \ln\left(\frac{\|\omega_0\|_2}{\rho_0}\right),
\end{equation}
where $\rho_0:=\frac{\sqrt{2}c_0^2}{ \nu } \fvi$.
\end{cor}

Using Lemma \ref{dugronwall}, we can prove a result similar to Lemma \ref{5}. More precisely, we have the following:
\begin{lemma}\label{l:7}
Let $\omega_0 \in \Lr^2$ and let $\omega_N^n$ be the solution of the numerical scheme
(\ref{Galerkin_scheme}). Also,  let $\dt\leq k_0$, with $k_0$ as in Corollary
\ref{C2}.
Then there exist constants $M_1=M_1(\nu, \fvi), N=N(\|\omega_0\|_2, \nu, \fvi)$ such that
\begin{equation}\label{59}
     \| \omega_N^n\|_{H^1} \le M_1,  \quad \forall n\ge N.
\end{equation}
\end{lemma}
\begin{proof}
        Taking the scalar product of (\ref{Galerkin_scheme}) with $-2\dt \Delta  \omega_N^{n+1}$,
we obtain
\begin{eqnarray}\label{60}
    \|\omega_N^{n+1}\|^2_{H^1} - \|\omega_N^{n}\|^2_{H^1}
     & +& \|\omega_N^{n+1} - \omega_N^{n}\|^2_{H^1} + 2 \nu \dt \|\Delta \omega_N^{n+1}\|_2^2 \nonumber\\
     & - &2  \dt \bb(\psi_N^n, \omega_N^n,\Delta  \omega_N^{n+1}) = -2\dt(f^n,\Delta  \omega_N^{n+1})_{L^2}.
\end{eqnarray}
Using the Cauchy--Schwarz inequality, we bound the right-hand side of~(\ref{60}) as
\begin{equation}\label{61}
- 2 \dt (f^n,\Delta  \omega_N^{n+1})_{L^2}
    \leq 2\dt \|f^n\|_2 \|\Delta  \omega_N^{n+1}\|_2 \leq \frac{\nu}{2} \dt \|\Delta  \omega_N^{n+1}\|_2^2 + \frac{2}{\nu} \dt \|f ^n\|_2^2.
\end{equation}
Using
the Wente type estimate (\ref{q:wente2}), the nonlinear term can be bounded as
\begin{eqnarray}\label{62}
2\dt \bb(\psi_N^n, \omega_N^n,\Delta  \omega_N^{n+1})  &=& 2\dt
\bb(\psi_N^n,   \omega_N^n-\omega_N^{n+1}, \Delta  \omega_N^{n+1})
+ 2\dt \bb(\psi_N^n,   \omega_N^{n+1}, \Delta  \omega_N^{n+1}) \nonumber\\
&  \leq& 2C_w \dt \|\nabla^\perp\psi_N^n\|_{H^1} \| \omega_N^{n+1}-\omega_N^{n}\|_{H^1} \|\Delta  \omega_N^{n+1}\|_2 \nonumber\\
&&+ 2C_w \dt \|\nabla^\perp\psi_N^n\|_{H^1} \| \omega_N^{n+1}\|_{H^1} \|\Delta  \omega_N^{n+1}\|_2 \nonumber\\
&\leq &\frac{1}{2} \|\omega_N^{n+1}-\omega_N^{n}\|^2_{H^1}+2C_w^2 \dt^2 \|\nabla^\perp\psi_N^n\|^2_{H^1}  \|\Delta  \omega_N^{n+1}\|_2^2\nonumber\\
&&+\frac{\nu}{2} \dt \|\Delta  \omega_N^{n+1}\|_2^2 + \frac{2C_w^2}{\nu} \dt
\|\nabla^\perp\psi_N^n\|^2_{H^1} \| \omega_N^{n+1}\|^2_{H^1} \nonumber\\
&\leq &\frac{1}{2} \|\omega_N^{n+1}-\omega_N^{n}\|^2_{H^1}+2C_w^2 \dt^2 \|\omega_N^n\|_2^2  \|\Delta  \omega_N^{n+1}\|_2^2\nonumber\\
&&+\frac{\nu}{2} \dt \|\Delta  \omega_N^{n+1}\|_2^2 + \frac{2C_w^2}{\nu} \dt
\|\omega_N^n\|_2^2 \| \omega_N^{n+1}\|^2_{H^1}.
\end{eqnarray}
Relations (\ref{60})--(\ref{62}) imply
\begin{equation}\label{63}\begin{aligned}
\left(1-\frac{2C_w^2}{\nu} \| \omega_N^{n}\|_2^2 \dt \right)&\|\omega_N^{n+1}\|^2_{H^1} - \|\omega_N^{n}\|^2_{H^1}
      + \frac{1}{2}\|\omega_N^{n+1} - \omega_N^{n}\|^2_{H^1}\\
	&+ \left( \nu -2C_w^2 \dt M_0^2\right)\dt \|\Delta \omega_N^{n+1}\|_2^2
	\leq \frac{2}{\nu} \dt \|f^n\|_2^2,
\end{aligned}\end{equation}
from which we find
\begin{equation}\label{64}
      \| \omega_N^{n+1}\|^2_{H^1} \leq \frac{1}{\alpha}\|\omega_N^{n}\|^2_{H^1}
         + \frac{2}{\alpha \nu} \dt \fvi^2,
\end{equation}
where
\begin{equation}\label{65}
      \alpha = 1-\frac{2C_w^2}{\nu} \dt M_0^2 >0.
\end{equation}

Now let $N_0= \lfloor T_0/\dt \rfloor$, with $T_0$ being given in
Corollary \ref{C2}, and for $r\ge 8 c_0^2/\nu$  arbitrarily fixed, let $N_r= \lfloor r/\dt \rfloor$. We are going to  apply Lemma
\ref{dugronwall} to \eqref{63}, with $\xi_n=\| \omega_N^{n}\|^2_{H^1}$,
$\eta_n=2C_w^2 \| \omega_N^{n-1}\|_2^2/\nu$, $\zeta_n=2\fvi^2/ \nu$,
$n_0=N_0+2$, $n_1=N_r-2$. For $n_2 \ge n_0$, we compute (taking into
account that, by \eqref{58}, $\| \omega_N^{n}\|_2^2 \leq 2\rho_0^2$,
for $n\geq N_0$):
\begin{eqnarray}
  &\dt& \sum_{n=n_2}^{n_2+n_1+1} \eta_n =  \dt \sum_{n=n_2}^{n_2+n_1+1} \frac{2C_w^2 }{\nu}\| \omega_N^{n-1}\|_2^2\le \frac{4C_w^2 }{\nu}\rho_0^2r:=a_1, \\
  &\dt& \sum_{n=n_2}^{n_2+n_1+1}  \zeta_n = \dt \sum_{n=n_2}^{n_2+n_1+1} \frac{2}{\nu}\fvi^2\leq\frac{2}{\nu}\fvi^2r:=a_2,\\
  &\dt& \sum_{n=n_2}^{n_2+n_1+1}  \xi_n=\dt \sum_{n=n_2}^{n_2+n_1+1}\| \omega_N^{n}\|^2_{H^1} \quad (\text{by } \eqref{49})\\
  && \quad \quad \quad \quad \quad \leq \frac{2}{\nu}\left(\|\omega_N^{n_2-1}\|_2^2 + \frac{c_0^2}{\nu} \fvi^2 (n_1+2) \dt\right) \quad (\text{by } \eqref{58})\\
  && \quad \quad \quad \quad \quad \leq \frac{2}{\nu}\left[2\rho_0^2 + \frac{c_0^2}{\nu} \fvi^2 r\right]=:a_3.
\end{eqnarray}
By \eqref{ugronest}, we obtain
\begin{eqnarray}
\|\omega_N^{n}\|^2_{H^1} &\leq& \left[ \frac{4}{\nu}\left(\frac{2\rho_0^2}{r} + \frac{1}{\nu\lb} \fvi^2 \right)+\frac{2}{\nu}\fvi^2r\right] \exp\left( \frac{16C_w^2 }{\nu}\rho_0^2r\right)\\
&=&:M_1^2(\nu, \fvi), \quad \forall n \geq N_0+N_r.
\end{eqnarray}
Taking $N=N_0+N_r$, we obtain conclusion \eqref{59} of Lemma \ref{l:7}.
\end{proof}
  \section{Collocation Fourier spectral approximation}
    Here we consider the collocation Fourier spectral spatial approximation of the scheme (\ref{scheme}). In order to maintain the long time stability of the fully discretized scheme, a common technique of using a modified form of the nonlinear term is utilized (see for instance \cite{T1966}). Moreover, we will use an alternative approach for the nonlinear analysis: instead of applying the Wente type estimate, we will use $\| \nabla \psi\|_{L^\infty}$, which is in turn bounded by $\|\psi\|_{H^3}^\epsilon \|\psi\|_{H^2}^{1-\epsilon}, \forall \epsilon \in (0, 1)$.  This alternative approach leads to a slightly more restrictive time step restriction for stability, but has the advantage of easy adaptance to the fully discrete collocation Fourier approximation.

\subsection{Fourier collocation spectral differentiation}

Consider a 2-D domain $\Omega=(0,L_x) \times (0,L_y)$.
For simplicity of presentation we assume that $L_x=L_y=L_0=1$ and
$L_x = N_x \cdot h_x$, $L_y=N_y \cdot h_y$ for some mesh sizes $h_x=h_y = h>0$ and some positive integers $N_x= N_y = 2N+1$.
All variables are evaluated at the regular numerical grid $(x_i, y_j)$,
with $x_i = i h$, $y_j=jh$, $0 \le i , j \le N$.

For a periodic function $f$ over the given 2-D numerical grid,
assume its discrete Fourier expansion is given by
\begin{equation}
  f_{i,j} = \sum_{k_1,l_1=-[N/2]}^{[N/2]}
   (\hat{f}_c^N)_{k_1,l_1} {\rm e}^ {2 \pi {\rm i} (k_1 x_i + l_1 y_j ) } .
   \label{spectral-coll-1}
\end{equation}
Note that $\hat{f}_c^N$ may not be the regular Fourier coefficients, due to the
aliasing error. In turn, its collocation interpolation operator becomes
\begin{eqnarray}
    {\cal I}_N f (\x)  = \sum_{k_1,l_1=-N}^{N}
   (\hat{f}_c^N)_{k_1,l_1}  {\rm e}^ {2 \pi {\rm i} (k_1 x + l_1 y ) } .
    \label{spectral-coll-projection-3}
\end{eqnarray}
As a result, its collocation Fourier spectral approximations to first and second order
partial derivatives (in $x$ direction) are given by
\begin{eqnarray}
  \left( {\cal D}_{Nx} f \right)_{i,j} = \sum_{k_1,l_1=-N}^{N}
   \left( 2 k_1 \pi {\rm i} \right) (\hat{f}_c^N)_{k_1,l_1}
   {\rm e}^ {2 \pi {\rm i} (k_1 x_i + l_1 y_j ) }  ,   \label{spectral-coll-2-1}
\\
  \left( {\cal D}_{Nx}^2 f \right)_{i,j} = \sum_{k_1,l_1=-[N/2]}^{[N/2]}
   \left( - 4 \pi^2 k_1^2 \right) \hat{f}_{k_1,l_1}
   {\rm e}^ {2 \pi {\rm i} (k_1 x_i + l_1 y_j ) }  . \label{spectral-coll-2-3}
\end{eqnarray}
The corresponding collocation spectral differentiations in $y$ directions can be
defined in the same way.
In turn, the discrete Laplacian, gradient and divergence can be denoted as
\begin{eqnarray}
  \Delta_N f =  \left( {\cal D}_{Nx}^2  + {\cal D}_{Ny}^2 \right) f ,
  \quad \nabla_N f = \left(  \begin{array}{c}
  {\cal D}_{Nx} f  \\
  {\cal D}_{Ny} f
  \end{array}  \right)  ,  \quad
  \nabla_N \cdot \left(  \begin{array}{c}
  f _1 \\
  f _2
  \end{array}  \right)  = {\cal D}_{Nx} f_1 + {\cal D}_{Ny} f_2 ,  \label{spectral-coll-3}
\end{eqnarray}
at the point-wise level.

  Moreover, given any periodic grid functions $f$ and $g$ (over the 2-D numerical grid),
the spectral approximations to the $L^2$ inner product and $L^2$ norm are introduced as
\begin{eqnarray}
  \left\| f \right\|_2 = \sqrt{ \left\langle f , f \right\rangle } ,  \quad \mbox{with} \quad
  \left\langle f , g \right\rangle  = h^2 \sum_{i,j=0}^{2N}   f_{i,j} g_{i,j} .
  \label{spectral-coll-inner product-1}
\end{eqnarray}
Meanwhile, such a discrete $L^2$ inner product can also be viewed in the Fourier space other than in physical space, with the help of Parseval equality:
\begin{equation}
   \left\langle f , g \right\rangle
   =   \sum_{k_1,l_1=-N}^{N}
   (\hat{f}_c^N)_{k_1,l_1}  \overline{(\hat{g}_c^N)_{k_1,l_1}}
   =   \sum_{k_1,l_1=-N}^{N}
   (\hat{g}_c^N)_{k_1,l_1}  \overline{(\hat{f}_c^N)_{k_1,l_1}}  ,
   \label{spectral-coll-inner product-2}
\end{equation}
in which $(\hat{f}_c^N)_{k_1,l_1}$, $(\hat{g}_c^N)_{k_1,l_1}$ are the Fourier
interpolation coefficients of the grid functions $f$ and  $g$ in the expansion as in
(\ref{spectral-coll-1}). Furthermore, a detailed calculation shows that the following
formulas of summation by parts are also valid at the discrete level:
\begin{eqnarray}
  \left\langle f ,  \nabla_N \cdot \left(  \begin{array}{c}
  g_1 \\
  g_2
  \end{array}  \right)  \right\rangle  = - \left\langle \nabla_N f ,  \left(  \begin{array}{c}
  g_1 \\
  g_2
  \end{array}  \right)  \right\rangle ,   \qquad
  \left\langle f ,  \Delta_N  g  \right\rangle
  = - \left\langle \nabla_N f ,  \nabla_N g   \right\rangle  .
  \label{spectral-coll-inner product-3}
\end{eqnarray}

\subsubsection{A preliminary estimate in Fourier collocation spectral space}

It is well-known that the existence of aliasing error in the nonlinear term
poses a serious challenge in the numerical analysis of Fourier collocation
spectral scheme.
To overcome a key difficulty associated with the $H^m$ bound of the nonlinear term
obtained by collocation interpolation, the following lemma is introduced.
The result is cited from a recent work \cite{GW11a}, and the detailed proof is skipped.

\begin{lemma}\label{spectral-coll-projection-4}
For any $\varphi \in P^{2N}$ in dimension $d$, we have
\begin{equation}
  \left\| {\cal I}_N \varphi \right\|_{H^k}
  \le  \left( \sqrt{2} \right)^d  \left\|  \varphi \right\|_{H^k} .
\end{equation}
\end{lemma}

In fact, an estimate for the $k=0$ case was reported in E's work
\cite{WE1992, WE1993}, with the constant given by $3^d$, while this lemma
sharpens the constant to $\sqrt{2}^d$. The case with $k > \frac{d}{2} = 1$
was covered in a classical approximation estimate for spectral expansions and
interpolations in Sobolev spaces, reported by Canuto and Quarteroni \cite{CQ1982}.
However,
due to the additional regularity requirement for interpolation operator  analysis,
the case of $k=1$ was not covered in any existing literature, which we require
for the $H^1$ bound of the nonlinear expansion in the global in time analysis.

\subsection{The first order semi-implicit scheme}

The fully discrete pseudo-spectral scheme follows the semi-implicit idea of
(\ref{scheme}) and (\ref{Galerkin_scheme}):

\begin{eqnarray}
  \frac{\omega^{n+1} - \omega^n}{\dt} + \frac12 \left( \u^n \DOT \nabla_N \omega^n
  + \nabla_N \cdot \left( \u^n \omega^n \right) \right)
  = \nu \Delta_N \omega^{n+1} + \f^n,   \label{scheme-coll-1st-1}
\\
  - \Delta_N \psi^{n+1} = \omega^{n+1} ,   \label{scheme-coll-1st-2}
\\
  \u^{n+1} = \nabla_N^{\bot} \psi^{n+1}
  = \left( {\cal D}_{Ny} \psi^{n+1} , {\cal D}_{Nx} \psi^{n+1} \right) .  \label{scheme-coll-1st-3}
\end{eqnarray}

  It is observed that the numerical velocity $\u^{n+1} = \nabla_N^{\bot} \psi^{n+1}$
is automatically divergence-free:
\begin{eqnarray}
  \nabla_N \cdot \u = {\cal D}_{Nx} u + {\cal D}_{Ny} v
  = - {\cal D}_{Nx} ( {\cal D}_{Ny} \psi) + {\cal D}_{Ny} ( {\cal D}_{Nx} \psi)
  =0 ,  \label{scheme-coll-1st-4}
\end{eqnarray}
at any time step. Meanwhile, note that the nonlinear term is a spectral approximation to
$\frac12 \u^n \DOT \nabla \omega$ and $\frac12 \nabla \cdot \left( \u \omega \right)$
at time step $t^n$. Furthermore, a careful application of summation by parts formula
(\ref{spectral-coll-inner product-3}) gives
\begin{eqnarray}
  \left\langle  \omega ,  \u \DOT \nabla_N \omega
  + \nabla_N \cdot \left( \u \omega \right)   \right\rangle
   = \left\langle  \omega ,  \u \DOT \nabla_N \omega  \right\rangle
  - \left\langle \nabla_N \omega , \u \omega   \right\rangle
  = 0  .  \label{scheme-coll-1st-5}
\end{eqnarray}
In other words, the nonlinear convection term appearing in the numerical scheme
(\ref{scheme-coll-1st-1}), so-called skew symmetric form, makes the nonlinear term
orthogonal to the vorticity field in the $L^2$ space, without considering the temporal
discretization. This property is crucial in the stability analysis for the Fourier
collocation spectral scheme (\ref{scheme-coll-1st-1})-(\ref{scheme-coll-1st-3}).

In addition, we denote $\U^n= (U^n, V^n)$,
$\pomega^n$ and $\ppsi^n$ as the continuous versions of
$\u^n$, $\omega^n$ and $\psi^n$,
respectively, with the formula given by (\ref{spectral-coll-projection-3}).
It is clear that $\U^n, \pomega^n, \ppsi^n \in P^N$
and the kinematic equation $\Delta \ppsi^n = \pomega^n$, $\U^n = \nabla^{\bot} \ppsi^n$
is satisfied at the continuous level. Because of these kinematic equations,
an application of elliptic regularity shows that
\begin{eqnarray}
  \left\| \ppsi^n \right\|_{H^{m+2}}  \le  C  \left\| \pomega^n \right\|_{H^m} ,  \quad
  \left\| \ppsi^n \right\|_{H^{m+2+\alpha}}
  \le  C  \left\| \pomega^n \right\|_{H^{m+\alpha}}  ,
  \label{est-coll-prelim-1}
\end{eqnarray}
in which we used the fact that all profiles have mean zero over the domain:
\begin{eqnarray}
   \overline{\ppsi^n} = 0 ,  \quad
  \overline{\U^n} = \left( - \overline{ \partial_y \ppsi^n} ,
  \overline{ \partial_x \ppsi^n} \right) = 0 ,  \quad
  \overline{\pomega^n} = \overline{\Delta \ppsi^n} = 0 .
  \label{est-coll-prelim-2}
\end{eqnarray}
Moreover, it is clear that the Poincar\'e inequality and elliptic regularity can
be applied because of this property.

 \begin{lemma}\label{collocation}
 Let $\omega_0 \in \Lr^2$ and let $\omega^n$ be the solution of the numerical scheme (\ref{scheme-coll-1st-1})-(\ref{scheme-coll-1st-3}). Also, let  $f \in L^\infty(\Real_+;H)$
 and set $\|f\|_{\infty} := \|f\|_{L^\infty(\Real+;H)}$.
  Then there exists $M_0 = M_0(\|\omega_0\|_2, \nu, \|f\|_{\infty})$ such that if
  \begin{equation} \label{constraint-coll-dt}
  \dt \leq \frac{\nu}{4C_w^2M_0^2},
  \end{equation}
   then
\begin{equation}\label{coll-est-L2-1}
   \|\pomega^n \|_{H^1} \leq M_0, \, \forall \, n \geq 0,
\end{equation}
\begin{equation}\label{coll-est-L2-2}
   \|\omega^n \|_{H^1}^2 \le \left(1+  \frac{\nu }{2c_0^2} \dt \right)^{-n} \|\pomega_0\|_{H^1}^2
           +  \frac{2c_0^4}{ \nu^2} \fvi^2 \left[ 1 - \left( 1 + \frac{\nu}{2c_0^2} \dt \right)^{-n} \right]
                ,
    \>\forall\, n\ge0,
\end{equation}
and
 \begin{eqnarray}\label{coll-est-L2-3}
     \frac{ \nu}{2} \dt \sum_{n=i}^{m} \|\pomega^{n}\|^2_{H^2}
     \leq \|\pomega_N^{i-1}\|_{H^1}^2 + \frac{c_0^2}{\nu} \fvi^2 (m-i+1) \dt, \,
             \quad \forall \, i=1, \cdots, m.
\end{eqnarray}
\end{lemma}

The proof of this lemma is organized as follows. First, an $H^\delta$ a-priori
assumption for the numerical solution $\omega^n$ is made. In turn, this assumption
leads to a global in time $L^2$ bound, with a standard application of
Sobolev embedding and
H\"older's inequality. However, this $L^2$ bound is not sufficient to recover the
a-priori assumption, due to the fact that the Wente type analysis is not
available for the collocation spectral approximation. Instead, a global in time
$H^1$ stability can also be derived with the help of the leading $L^2$ bound.
Moreover, both the global in time $L^2$ and $H^1$ bound constants are independent
of the a-priori constant $\tilde{C}_1$. As a result, the a-priori assumption can be
recovered so that an induction can be applied to established the above lemma.

\subsection{Leading estimate: $L^\infty (0,T; L^2) \cap L^2 (0,T; H^1)$
estimate for $\omega$}

  Assume a-priori that
\begin{equation}
  \| \pomega^n \|_{H^\delta} \le \tilde{C}_1 ,  \quad
  \mbox{$\pomega^n$ is the continuous version of $\omega^n$} , \label{est-coll-a priori}
\end{equation}
for some $\delta > 0$ at time step $t^n$. Note that
$\tilde{C}_1$ is a global constant in time.
We are going to prove that such a bound
for the numerical solution is also available at time step $t^{n+1}$.

Taking the discrete inner product of (\ref{scheme-coll-1st-1}) with $2 \dt \omega^{n+1}$ gives
\begin{eqnarray}
  & &
  \|  \omega^{n+1}  \|_2^2  -  \|  \omega^n  \|_2^2
  +  \|  \omega^{n+1}  -  \omega^n  \|_2^2
  +  2 \nu \dt \|  \nabla_N \omega^{n+1}  \|_2^2    \nonumber
\\
  & &
  = - \dt \left\langle \u^n \DOT \nabla_N \omega^n
  + \nabla_N \cdot \left( \u^n \omega^n \right) , \omega^{n+1} \right\rangle
  + 2 \dt \left\langle  \f^n , \omega^{n+1} \right\rangle  ,  \label{est-coll-L2-1}
\end{eqnarray}
in which the summation by parts formula (\ref{spectral-coll-inner product-3}) was
applied to the diffusion term.
A bound for the outer force term is straightforward:
\begin{eqnarray}
  2 \left\langle  \f^n , \omega^{n+1} \right\rangle
  &\le& 2 \left\| f^n \right\|_2  \cdot \left\| \omega^{n+1} \right\|_2
  \le 2 C_2 \left\| f^n \right\|_2  \cdot \left\| \nabla_N \omega^{n+1} \right\|_2   \nonumber
\\
  &\le&
  \frac{\nu}2  \left\| \nabla_N \omega^{n+1} \right\|_2^2
  + \frac{2 C_2^2}{\nu} \left\| f^n \right\|_2^2
  \le   \frac{\nu}2  \left\| \nabla_N \omega^{n+1} \right\|_2^2 + \frac{2 C_2^2 M^2}{\nu}  ,
  \label{est-coll-L2-2}
\end{eqnarray}
in which a Poincar\'e inequality
\begin{equation}
    \left\| \omega^{n+1} \right\|_2  \le C_2    \left\| \nabla_N \omega^{n+1} \right\|_2 ,
    \label{est-coll-L2-3}
\end{equation}
was used in the third step. For the nonlinear term, we start with the following
rewritten form:
\begin{eqnarray}
  & &
   - \dt \left\langle \u^n \DOT \nabla_N \omega^n
  + \nabla_N \cdot \left( \u^n \omega^n \right) , \omega^{n+1} \right\rangle  \nonumber
\\
   &=&
     - \dt \left\langle \u^n \DOT \nabla_N \omega^{n+1}
  + \nabla_N \cdot \left( \u^n \omega^{n+1} \right) , \omega^{n+1} \right\rangle  \nonumber
\\
  & &
     + \dt \left\langle \u^n \DOT \nabla_N ( \omega^{n+1} - \omega^n )
  + \nabla_N \cdot \left( \u^n ( \omega^{n+1} - \omega^n) \right) , \omega^{n+1} \right\rangle  .
   \label{est-coll-L2-4}
\end{eqnarray}
The first term disappears, using a similar analysis as (\ref{scheme-coll-1st-5}):
\begin{equation}\begin{aligned}
  &\left\langle  \u^n \DOT \nabla_N \omega^{n+1}
  + \nabla_N \cdot \left( \u^n \omega^{n+1} \right) , \omega^{n+1}  \right\rangle\\
   &\hbox to80pt{} {}= \left\langle  \omega^{n+1} ,  \u^n \DOT \nabla_N \omega^{n+1}  \right\rangle
  - \left\langle \nabla_N \omega^{n+1} , \u^n \omega^{n+1}   \right\rangle
  = 0  .  \label{est-coll-L2-5}
\end{aligned}\end{equation}
For the second term, the summation by parts formula
(\ref{spectral-coll-inner product-3}) can be applied:
\begin{eqnarray}
  \left\langle \u^n \DOT \nabla_N ( \omega^{n+1} - \omega^n) , \omega^{n+1} \right\rangle
  &=&
   - \left\langle \omega^{n+1} - \omega^n ,
   \nabla_N \cdot (\u^n \omega^{n+1}) \right\rangle  ,   \label{est-coll-L2-6-1}
\\
  \left\langle \nabla_N \cdot \left( \u^n ( \omega^{n+1} - \omega^n) \right\rangle ,
  \omega^{n+1} \right\rangle
  &=&
  - \left\langle \omega^{n+1} - \omega^n , \u^n \cdot \nabla_N \omega^{n+1} \right\rangle  ,
   \label{est-coll-L2-6-2}
\end{eqnarray}
For the term $\nabla_N \cdot (\u^n \omega^{n+1})$, we note that it cannot be expanded
as $\u^n \cdot \nabla_N \omega^{n+1}$, as in the Fourier-Galerkin approximation,
even though $\u^n$ is divergence-free at the discrete level (\ref{scheme-coll-1st-4}).
In the collocation space, we have to start from
\begin{equation}
  \nabla_N \cdot (\u^n \omega^{n+1})
  = {\cal D}_{Nx} (u^n \omega^{n+1})  + {\cal D}_{Ny} (v^n \omega^{n+1}).
  \label{est-coll-L2-6-3}
\end{equation}
To obtain an estimate of these nonlinear expansions, we recall that
$\U^n= (U^n, V^n)$, $\pomega^{n+1}$ and $\ppsi^{n+1}$ are the
continuous versions of $\u^n$, $\omega^{n+1}$ and $\psi^{n+1}$,
respectively. Since $\U^n, \pomega^{n+1} \in P^N$, we have
$\U^n \pomega^{n+1} \in P^{2N}$ and an application of Lemma
\ref{spectral-coll-projection-4} indicates that
\begin{eqnarray}
  \left\|  {\cal D}_{Nx} (u^n \omega^{n+1})  \right\|_2
  &=& \left\|  \partial_x {\cal I}_N ( U^n \omega^{n+1})  \right\|_{2}
  \le 2  \left\|  \partial_x ( U^n \pomega^{n+1})  \right\|_{2}  ,  \nonumber
\\
  \left\|  {\cal D}_{Ny} (v^n \omega^{n+1})  \right\|_2
  &=& \left\|  \partial_y {\cal I}_N ( V^n \omega^{n+1})  \right\|_{2}
  \le 2  \left\|  \partial_y ( V^n \pomega^{n+1})  \right\|_{2}   .   \label{est-coll-L2-6-4}
\end{eqnarray}
Subsequently, a detailed expansion in the continuous space and an application of
H\"older's inequality show that
\begin{eqnarray}
   \left\|  \partial_x ( U^n \pomega^{n+1})  \right\|_{2}
   &=&  \left\|  U^n_x \pomega^{n+1} +  U^n \pomega_x^{n+1}  \right\|_{2}
   \le \left\|  U^n_x \pomega^{n+1} \right\|_{2}
   + \left\| U^n \pomega_x^{n+1}  \right\|_{2}   \nonumber
\\
  &\le&
   \left\|  U^n_x  \right\|_{L^{2/(1-\delta)}}  \cdot  \left\| \pomega^{n+1} \right\|_{L^{2/\delta}}
   + \left\| U^n  \right\|_{L^\infty} \cdot  \left\| \pomega_x^{n+1}  \right\|_{2}  .
   \label{est-coll-L2-6-5}
\end{eqnarray}
Furthermore, a 2-D Sobolev embedding gives
\begin{eqnarray}
    \left\|  U^n_x  \right\|_{L^{2/(1-\delta)}} \left\| \pomega^{n+1} \right\|_{L^{2/\delta}}
   \le  C  \left\|  U^n_x  \right\|_{H^\delta}
    \left\| \pomega^{n+1} \right\|_{H^1}
   \le  C  \left\|  \pomega^n  \right\|_{H^\delta}
    \left\| \nabla \pomega^{n+1} \right\|_{2}\,,  \label{est-coll-L2-6-6}
\end{eqnarray}
in which  the elliptic regularity (\ref{est-coll-prelim-1}) and the Poincar\'e inequality were
utilized in the last step. The second part in (\ref{est-coll-L2-6-5}) can be handled in
a straightforward way:
\begin{eqnarray}
   \left\| U^n  \right\|_{L^\infty} \cdot  \left\| \pomega_x^{n+1}  \right\|_{2}
   \le   C \left\| U^n  \right\|_{H^{1+\delta}} \cdot  \left\| \nabla \pomega^{n+1}  \right\|_{2}
   \le   C \left\| \pomega^n  \right\|_{H^\delta} \cdot
   \left\| \nabla \pomega^{n+1}  \right\|_{2}  ,  \label{est-coll-L2-6-7}
\end{eqnarray}
with the the elliptic regularity (\ref{est-coll-prelim-1}) applied again in the second step.
A combination of (\ref{est-coll-L2-6-6}) and  (\ref{est-coll-L2-6-7}) yields
\begin{eqnarray}
  \left\|  \partial_x ( U^n \pomega^{n+1})  \right\|_{2}
  \le  C \left\| \pomega^n  \right\|_{H^\delta} \cdot
   \left\| \nabla \pomega^{n+1}  \right\|_{2}  .  \label{est-coll-L2-6-8}
\end{eqnarray}
Similar estimates can be derived for $\left\|  \partial_y ( V^n \pomega^{n+1})  \right\|_{2}$.
Going back to (\ref{est-coll-L2-6-4}), we arrive at
\begin{eqnarray}
  \left\|  \nabla_N \cdot (\u^n \omega^{n+1})   \right\|_2
  \le  C \left\| \pomega^n  \right\|_{H^\delta} \cdot
   \left\| \nabla \pomega^{n+1}  \right\|_{2}
   =  C \left\| \pomega^n  \right\|_{H^\delta} \cdot
   \left\| \nabla_N \omega^{n+1}  \right\|_2 ,  \label{est-coll-L2-6-9}
\end{eqnarray}
in which the second step is based on the fact that $\pomega^n, \pomega^{n+1} \in P^N$,
so that the corresponding $L^2$ and $H^\delta$ norms are equivalent between
the continuous projection and the discrete version. In addition, the nonlinear term
in (\ref{est-coll-L2-6-2}) can be controlled in a similar way:
\begin{eqnarray}
  \left\| \u^n \cdot \nabla_N \omega^{n+1} \right\|_2
  \le  \left\| \u^n \right\|_{\infty}  \cdot \left\| \nabla_N \omega^{n+1} \right\|_2
  =   C \left\| \pomega^n  \right\|_{H^\delta} \cdot
   \left\| \nabla_N \omega^{n+1}  \right\|_2 ,   \label{est-coll-L2-6-10}
\end{eqnarray}
with a discrete Sobolev imbedding inequality applied in the second step. Therefore, a substitution of
(\ref{est-coll-L2-6-9})--(\ref{est-coll-L2-6-10}) into (\ref{est-coll-L2-4}), (\ref{est-coll-L2-5}),  (\ref{est-coll-L2-6-1})--(\ref{est-coll-L2-6-2}) results in
\begin{align}
   - \dt \bigl\langle \u^n \DOT \nabla_N \omega^n
  &+ \nabla_N \cdot \left( \u^n \omega^n \right) , \omega^{n+1} \bigr\rangle\\
    &\le
    C \dt  \| \pomega^n \|_{H^\delta}  \cdot
    \left\| \omega^{n+1} - \omega^n \right\|_2 \cdot
    \left\|  \nabla_N \omega^{n+1}  \right\|_2   \nonumber
\\
   &\le
    C \tilde{C}_1 \dt
    \left\| \omega^{n+1} - \omega^n \right\|_2 \cdot
    \left\|  \nabla_N \omega^{n+1}  \right\|_2   \nonumber
\\
   &\le
    \frac12 \nu \dt \left\|  \nabla_N \omega^{n+1}  \right\|_2^2
    + \frac{C_3 \tilde{C}_1^2}{\nu} \dt
    \left\| \omega^{n+1} - \omega^n \right\|_2^2 .
       \label{est-coll-L2-7}
\end{align}
Its combination with (\ref{est-coll-L2-2}), (\ref{est-coll-L2-4}), (\ref{est-coll-L2-5})
and (\ref{est-coll-L2-1}) leads to
\begin{equation}
  \|  \omega^{n+1}  \|_2^2  -  \|  \omega^n  \|_2^2
  +  \left( 1 - \frac{C_3 \tilde{C}_1^2}{\nu} \dt  \right) \|  \omega^{n+1}  -  \omega^n  \|_2^2
  +  \nu \dt \|  \nabla_N \omega^{n+1}  \|_2^2
  \le  \frac{2 C_2^2 M^2}{\nu} \dt  .  \label{est-coll-L2-8}
\end{equation}
Under a constraint for the time step
\begin{equation}
   \frac{C_3 \tilde{C}_1^2}{\nu} \dt  \le \frac12 ,  \quad \mbox{i.e.}, \quad
   \dt \le  \frac{\nu}{2 C_3 \tilde{C}_1^2}  ,  \label{constraint-coll-dt-1}
\end{equation}
we arrive at
\begin{equation}\label{est-coll-L2-9}
   \|  \omega^{n+1}  \|_2^2  -  \|  \omega^n  \|_2^2
  +  \frac12 \|  \omega^{n+1}  -  \omega^n  \|_2^2
  +  \nu \dt \|  \nabla_N \omega^{n+1}  \|^2_2
  \le  C_4 \dt
\end{equation}
with $C_4 = ({2 C_2^2 M^2})/{\nu}$.
Furthermore, an application of the Poincar\'e inequality (\ref{est-coll-L2-3}) implies that
\begin{equation}
   \|  \omega^{n+1}  \|_2^2  -  \|  \omega^n  \|_2^2
  +  C_5 \nu \dt \|   \omega^{n+1}  \|_2^2
  \le  C_4 \dt ,  \quad  \mbox{with}  \, \, \,
  C_5 = \frac{1}{C_2^2}.  \label{est-coll-L2-10}
\end{equation}
Applying an induction argument to the above estimate yields
\begin{align}
   &\|  \omega^{n+1}  \|_2^2 \le (1+ C_5 \nu \Delta t)^{-(n+1)} \|  \omega^0  \|_2^2
   + \frac{C_4}{C_5 \nu}\notag\\
   \Rightarrow\quad
   &\|  \omega^{n+1}  \|_2  \le (1+ C_5 \nu \Delta t)^{-({n+1})/{2} }  \|  \omega^0  \|_2
   + \sqrt{\frac{C_4}{C_5 \nu}} := C_6 .   \label{est-coll-L2-11}
\end{align}
Note that $C_6$ is a time dependent value; however, its time dependence is
in exponential decay so that a global in time bound is available.

  In addition, we also have the $L^2 (0, T; H^1)$ bound for the
numerical solution:
\begin{equation}
  \nu \dt \sum_{k=i+1}^{N_k}   \left\|  \nabla_N  \omega^k  \right\|_2^2
  \le   \|  \omega^i  \|_2^2  +  C_4  \, \left( T^*- t^i \right) .  \label{est-coll-L2-12}
\end{equation}

  However, it is observed that the a-priori estimate (\ref{est-coll-L2-11}) is not sufficient
to bound the $H^\delta$ norm (\ref{est-coll-a priori}) of the vorticity field.
In turn, we perform a higher order
energy estimate $L^\infty (0,T;  H^1) \cap L^2 (0, T; H^2)$ for the numerical solution
of the vorticity field.

\subsection{$L^\infty (0, t_1;  H^1) \cap L^2 (0, t_1; H^2)$ estimate for $\omega$}

Taking the inner product of (\ref{scheme-coll-1st-1}) with $- 2 \dt \Delta_N \omega^{n+1}$ gives
\begin{eqnarray}
  & &
  \|  \nabla_N \omega^{n+1}  \|_2^2  -  \|  \nabla_N \omega^n  \|_2^2
  +  \|  \nabla_N \left( \omega^{n+1}  -  \omega^n \right) \|_2^2
  +  2 \nu \dt \|  \Delta_N \omega^{n+1}  \|_2^2    \nonumber
\\
  & &
  =  \dt \left\langle \u^n \DOT \nabla_N \omega^n
  + \nabla_N \cdot \left( \u^n \omega^n \right) , \Delta_N \omega^{n+1} \right\rangle
  - 2 \dt \left\langle  \f^n , \Delta_N \omega^{n+1} \right\rangle  .  \label{est-coll-H1-1}
\end{eqnarray}
The Cauchy inequality can be applied to bound the outer force term:
\begin{align}
  - 2 \left\langle  \f^n , \Delta_N \omega^{n+1} \right\rangle
  &\le \frac12 \nu  \left\| \Delta_N \omega^{n+1} \right\|_2^2
  + \frac{2}{\nu} \left\| \f^n \right\|_2^2\notag\\
  &\le   \frac12 \nu  \left\| \Delta_N \omega^{n+1} \right\|_2^2 + \frac{2 M^2}{\nu}  .
  \label{est-coll-H1-2}
\end{align}
For the nonlinear terms, we first make the following decomposition:
\begin{align}
  &\u^n \DOT \nabla_N \omega^n
  =   - \u^n \DOT \nabla_N \left( \omega^{n+1} - \omega^n \right)
  - \left( \u^{n+1} - \u^n \right) \DOT \nabla_N \omega^{n+1}\notag\\
  &\hbox to155pt{} {}+ \u^{n+1} \DOT \nabla_N \omega^{n+1} , \label{est-coll-H1-3-1}\\
  &\nabla_N \cdot \bigl( \u^n \omega^n \bigr)
   =   \nabla_N \cdot \bigl(  - \u^n ( \omega^{n+1} - \omega^n )
  - ( \u^{n+1} - \u^n )  \omega^{n+1}\notag\\
  &\hbox to155pt{} {}+  \u^{n+1} \omega^{n+1}  \bigr) . \label{est-coll-H1-3-2}
\end{align}
For the first term, the a-priori assumption (\ref{est-coll-a priori}) gives
\begin{align}
   \left\|  - \u^n \DOT \nabla_N \left( \omega^{n+1} - \omega^n \right)  \right\|_2
   &\le  \left\|  \u^n \right\|_{\infty} \cdot
   \left\| \nabla_N \left( \omega^{n+1} - \omega^n \right)  \right\|_2\notag\\
    &\le  C \tilde{C}_1  \left\| \nabla_N \left( \omega^{n+1} - \omega^n \right)  \right\|_2  ,
    \label{est-coll-H1-4}
\end{align}
in which we applied the discrete Sobolev inequality in 2-D:
$\| \u^n \|_{\infty} \le C \| \u^n \|_{H_h^{1+\delta}} \le C \| \omega^n \|_{H_h^\delta}$.
This in turn leads to
\begin{align}
    \dt \bigl\langle - \u^n \DOT \nabla_N &( \omega^{n+1} - \omega^n) ,
    \Delta_N \omega^{n+1} \bigr\rangle  \nonumber
\\
    &\le
    C \tilde{C}_1 \dt   \left\| \nabla_N \left( \omega^{n+1} - \omega^n \right)  \right\|_2
    \cdot \left\|  \Delta_N \omega^{n+1}  \right\|_2   \nonumber
\\
   &\le
    \frac14 \nu \dt \left\|  \Delta_N \omega^{n+1}  \right\|_2^2
    + \frac{C \tilde{C}_1^2}{\nu} \dt
    \left\| \nabla_N \left( \omega^{n+1} - \omega^n \right) \right\|_2^2 .
       \label{est-coll-H1-5}
\end{align}
The conservative nonlinear term $\nabla_N \cdot \left(
\u^n ( \omega^{n+1} - \omega^n )  \right)$ can be analyzed
as in (\ref{est-coll-L2-6-3})--(\ref{est-coll-L2-6-10}):
\begin{eqnarray}
   & & \left\|  \nabla_N \cdot \left(   \u^n ( \omega^{n+1} - \omega^n )  \right)  \right\|_2
   \le
      \left\|  {\cal D}_{Nx}  \left( u^n ( \omega^{n+1} - \omega^n )  \right)  \right\|_2
 +   \left\|  {\cal D}_{Ny}  \left( v^n ( \omega^{n+1} - \omega^n )  \right)  \right\|_2  \nonumber
\\
  & &  \qquad
   \le 2  \left( \left\|  \partial_x  \left( U^n ( \pomega^{n+1} - \pomega^n )  \right)  \right\|_{2}
   + \left\|  \partial_y  \left( V^n ( \pomega^{n+1} - \pomega^n )  \right)  \right\|_{2} \right)  ,
   \label{est-coll-H1-4-2}
\\
  & &
   \left\|  \partial_x \left( U^n ( \pomega^{n+1} - \pomega^n )  \right)  \right\|_{2}
   =  \left\|  U^n_x ( \pomega^{n+1} - \pomega^n )
   +  U^n ( \pomega^{n+1} - \pomega^n )_x  \right\|_{2}    \nonumber
\\
  & &  \qquad
   \le \left\|  U^n_x  \right\|_{L^{2/(1-\delta)}}
   \cdot  \left\| \pomega^{n+1} - \pomega^n \right\|_{L^{2/\delta}}
   + \left\| U^n  \right\|_{L^\infty}
   \cdot  \left\| ( \pomega^{n+1} - \pomega^n )_x \right\|_{2}    \nonumber
\\
  & &  \qquad
   \le C \left\|  \pomega^n  \right\|_{H^\delta}
   \cdot  \left\| \nabla ( \pomega^{n+1} - \pomega^n ) \right\|_{2}
   \le C \tilde{C}_1
   \left\| \nabla_N ( \omega^{n+1} - \omega^n ) \right\|_2 ,
   \label{est-coll-H1-4-3}
\\
  & &
   \left\|  \partial_y \left( V^n ( \pomega^{n+1} - \pomega^n )  \right)  \right\|_{2}
   \le C \tilde{C}_1
   \left\| \nabla_N ( \omega^{n+1} - \omega^n ) \right\|_2 ,
   \label{est-coll-H1-4-4}
\end{eqnarray}
with the help of the elliptic regularity (\ref{est-coll-prelim-1}), Poincar\'e's inequality
and 2-D Sobolev embedding. Consequently, we see that the first part of the nonlinear
term  (\ref{est-coll-H1-3-2}) has the same bound as  (\ref{est-coll-H1-4}):
\begin{eqnarray}
   \left\|  \nabla_N \cdot \left( \u^n ( \omega^{n+1} - \omega^n ) \right)  \right\|_2
    \le  C \tilde{C}_1  \left\| \nabla_N \left( \omega^{n+1} - \omega^n \right)  \right\|_2,
    \label{est-coll-H1-4-5}
\end{eqnarray}
which in turn leads to an estimate  similar to (\ref{est-coll-H1-5}):
\begin{eqnarray}
    & & \dt \left\langle -  \nabla_N  \cdot \left( \u^n ( \omega^{n+1} - \omega^n)  \right) ,
    \Delta_N \omega^{n+1} \right\rangle  \nonumber
\\
    &\le&
    \frac14 \nu \dt \left\|  \Delta_N \omega^{n+1}  \right\|_2^2
    + \frac{C \tilde{C}_1^2}{\nu} \dt
    \left\| \nabla_N \left( \omega^{n+1} - \omega^n \right) \right\|_2^2 .
       \label{est-coll-H1-5-2}
\end{eqnarray}

  For the second term in (\ref{est-coll-H1-3-1}), we start with the following Sobolev inequality:
\begin{eqnarray}
  \left\|   \nabla_N \omega^{n+1}  \right\|_2
  &=& \left\|   \nabla \pomega^{n+1}  \right\|_{2}
  \le \left\|   \pomega^{n+1}  \right\|_{H^1}
  \le C \left\|  \pomega^{n+1}  \right\|^{1/2}_2
   \cdot  \left\|  \pomega^{n+1}  \right\|_{H^2}^{1/2}   \nonumber
\\
  &\le&
  C \left\|  \pomega^{n+1}  \right\|^{1/2}_2
  \cdot  \left\|  \Delta \pomega^{n+1}  \right\|^{1/2}_2
  \le C C_6^{1/2} \left\|  \Delta \pomega^{n+1}  \right\|^{1/2}_2,
  \label{est-coll-H1-6}
\end{eqnarray}
in which an elliptic regularity $ \left\|  \pomega^{n+1}  \right\|_{H^2}
  \le C  \left\|  \Delta \pomega^{n+1}  \right\|_2$  was utilized in the second step
and the leading $L^2$ estimate (\ref{est-coll-L2-11}) was used in the last step.
Similarly, we also observe that the kinematic relationships
\begin{equation}
  \U^{n+1} - \U^n = \nabla^{\bot} \left( \ppsi^{n+1} - \ppsi^n \right) ,  \quad
  \Delta  \left( \ppsi^{n+1} -  \ppsi^n  \right)  =  \pomega^{n+1} - \pomega^n ,
  \label{est-coll-H1-7}
\end{equation}
indicate the following Sobolev estimates:
\begin{eqnarray}
  \left\|  \u^{n+1} - \u^n  \right\|_{\infty}
  &\le& \left\|  \U^{n+1} - \U^n  \right\|_{L^\infty}    \nonumber
\\
  &\le&
  C \left\|  \U^{n+1} - \U^n  \right\|_{H^{1+\delta}}
  \le  C \left\|  \ppsi^{n+1} - \ppsi^n  \right\|_{H^{2+\delta}}
  \le  C \left\|  \pomega^{n+1} - \pomega^n  \right\|_{H^{\delta}}   \nonumber
\\
  &\le&
  C \left\|  \pomega^{n+1} - \pomega^n  \right\|^{1-\delta}_2
  \left\|  \pomega^{n+1} - \pomega^n  \right\|_{H^1}^{\delta}   \nonumber
\\
  &\le&
  C \left\|  \pomega^{n+1} - \pomega^n  \right\|^{1-\delta}_2
  \left\|  \nabla \left( \pomega^{n+1} - \pomega^n \right) \right\|^{\delta}_2  \nonumber
\\
  &\le&
  C \left( 2 C_6 \right)^{1-\delta}
  \left\|  \nabla \left( \pomega^{n+1} - \pomega^n \right) \right\|^{\delta}_2,
  \label{est-coll-H1-8}
\end{eqnarray}
in which estimate (\ref{est-coll-L2-11}) was used in the last step.
Consequently, a combination of (\ref{est-coll-H1-6}) and (\ref{est-coll-H1-8}) indicates that
\begin{align}
  \bigl\| ( \u^{n+1} &- \u^n ) \DOT \nabla_N \omega^{n+1}  \bigr\|_2
  \le
  \left\| \u^{n+1} - \u^n  \right\|_{\infty}
    \cdot \left\|  \nabla_N \omega^{n+1}  \right\|_2    \nonumber
\\
  &\le
  C C_6^{1/2} \left( 2 C_6 \right)^{1-\delta}
  \left\|  \nabla \left( \pomega^{n+1} - \pomega^n \right) \right\|^{\delta}_2
  \cdot  \left\|  \Delta \pomega^{n+1}  \right\|^{1/2}_2    \nonumber
\\
  &\le
  C C_6^{1/2} \left( 2 C_6 \right)^{1-\delta}
  \left\|  \nabla_N \left( \omega^{n+1} - \omega^n \right) \right\|_2^{\delta}
  \cdot  \left\|  \Delta_N \omega^{n+1}  \right\|_2^{1/2} ,
  \label{est-coll-H1-9}
\end{align}
due to the fact that $\pomega \in P^N$.
In turn, the following estimate is obtained
\begin{align}
  &\dt  \left\langle  - \left( \u^{n+1} - \u^n \right) \DOT \nabla_N \omega^{n+1} ,
  \Delta_N \omega^{n+1}  \right\rangle\notag\\
  &\hbox to50pt{} {}\le C C_6^{3/2} \dt  \left\|  \nabla_N \left( \omega^{n+1} - \omega^n \right) \right\|_2^{\delta}
  \cdot  \left\|  \Delta_N \omega^{n+1}  \right\|_2^{3/2} .
   \label{est-coll-H1-10}
\end{align}
Meanwhile, the second conservative nonlinear term in (\ref{est-coll-H1-3-2}),
$\nabla_N \cdot \left(  ( \u^{n+1} - \u^n )  \omega^{n+1}  \right)$, can be expanded
and analyzed in a similar way:
\begin{eqnarray}
   & & \left\|  \nabla_N \cdot \left(   ( \u^{n+1} - \u^n )  \omega^{n+1}  \right)  \right\|_2
   \le
      \left\|  {\cal D}_{Nx}  \left( ( u^{n+1} - u^n ) \omega^{n+1}  \right)  \right\|_2
 +   \left\|  {\cal D}_{Ny}  \left( ( v^{n+1} - v^n ) \omega^{n+1}  \right)  \right\|_2  \nonumber
\\
  & &  \qquad
   \le 2  \left( \left\|  \partial_x  \left( ( U^{n+1} - U^n )  \pomega^{n+1}   \right)  \right\|_{2}
   + \left\|  \partial_y  \left( ( V^{n+1} - V^n )  \pomega^{n+1}   \right)  \right\|_{2} \right)  ,
   \label{est-coll-H1-10-2}
\\
  & &
   \left\|  \partial_x \left( ( U^{n+1} - U^n )  \pomega^{n+1}  \right)  \right\|_{2}
   =  \left\|  ( U^{n+1} - U^n )_x  \pomega^{n+1}
   +  ( U^{n+1} - U^n ) \pomega_x^{n+1}  \right\|_{2}    \nonumber
\\
  & &  \qquad
   \le \left\|  ( U^{n+1} - U^n )_x  \right\|_{L^{2/(1-\delta)}}
   \cdot  \left\| \pomega^{n+1}  \right\|_{L^{2/\delta}}
   + \left\| U^{n+1} - U^n  \right\|_{L^\infty}
   \cdot  \left\| \pomega_x^{n+1}  \right\|_{2}    \nonumber
\\
  & &  \qquad
   \le C \left\|  U^{n+1} - U^n  \right\|_{H^{1+\delta}}
   \cdot  \left\| \nabla \pomega^{n+1}  \right\|_{2}   \nonumber
\\
  & & \qquad
  \le C C_6^{3/2}  \left\|  \nabla_N \left( \omega^{n+1} - \omega^n \right) \right\|_2^{\delta}
  \cdot  \left\|  \Delta_N \omega^{n+1}  \right\|_2^{1/2}  ,
   \label{est-coll-H1-10-3}
\\
  & &
  \left\|  \partial_y \left( ( V^{n+1} - V^n )  \pomega^{n+1}  \right)  \right\|_{2}
  \le C C_6^{3/2}  \left\|  \nabla_N \left( \omega^{n+1} - \omega^n \right) \right\|_2^{\delta}
  \cdot  \left\|  \Delta_N \omega^{n+1}  \right\|_2^{1/2}  .
   \label{est-coll-H1-10-4}
\end{eqnarray}
Again, the elliptic regularity (\ref{est-coll-prelim-1}), Poincar\'e's inequality
and 2-D Sobolev embedding were repeatedly used in the analysis.
As a result, its combination with (\ref{est-coll-H1-10}) leads to
\begin{eqnarray}
   & & \dt  \left\langle  - \left( \u^{n+1} - \u^n \right) \DOT \nabla_N \omega^{n+1}
   -  \nabla_N \cdot \left( ( \u^{n+1} - \u^n ) \omega^{n+1} \right) ,
  \Delta_N \omega^{n+1}  \right\rangle   \nonumber
\\
  &\le&
  C_7 C_6^{3/2} \dt  \left\|  \nabla_N \left( \omega^{n+1} - \omega^n \right) \right\|_2^{\delta}
  \cdot  \left\|  \Delta_N \omega^{n+1}  \right\|_2^{3/2} .
   \label{est-coll-H1-10-5}
\end{eqnarray}
We can always choose $0 < \delta < \frac12$, so that an application of
Young's inequality ($a b \le \frac{a^p}{p} + \frac{b^q}{q}$
with $\frac{1}{p} + \frac{1}{q} =1$) gives
\begin{eqnarray}
  & &
  \left\|  \nabla_N \left( \omega^{n+1} - \omega^n \right) \right\|_2^{\delta}
  \cdot  \left\|  \Delta_N \omega^{n+1}  \right\|_2^{3/2}
  \le  C_8  \left\|  \nabla_N \left( \omega^{n+1} - \omega^n \right) \right\|_2^{4 \delta}
  +  \frac{\nu}{2 C_7 C_6^{3/2}}  \left\|  \Delta_N \omega^{n+1}  \right\|_2^2  , \nonumber
\\
  & &
  \mbox{with}  \quad
  C_8  = \frac14  \left( \frac{3 C_7 C_6^{3/2} }{2 \nu} \right)^3 .
  \label{est-coll-H1-11}
\end{eqnarray}
Furthermore, since $4 \delta < 2$, we can apply Young's inequality to
$\left\| \nabla_N \left( \omega^{n+1} - \omega^n \right) \right\|_2^{4 \delta}$ and obtain
\begin{eqnarray}
  C_8  \left\|  \nabla_N \left( \omega^{n+1} - \omega^n \right) \right\|_2^{4 \delta}
  \le \frac{1}{C_7 C_6^{3/2}}
  \left\|  \nabla_N \left( \omega^{n+1} - \omega^n \right) \right\|_2^2  + C_9 ,
  \label{est-coll-H1-12}
\end{eqnarray}
in which $C_9$ depends on $C_6$, $C_7$, $C_8$ and $\delta$. As a result, substituting
(\ref{est-coll-H1-11})--(\ref{est-coll-H1-12}) into (\ref{est-coll-H1-10}) gives an estimate for the
second nonlinear term:
\begin{align}
  &\dt  \left\langle  - \left( \u^{n+1} - \u^n \right) \DOT \nabla_N \omega^{n+1}
   -  \nabla_N \cdot \left( ( \u^{n+1} - \u^n ) \omega^{n+1} \right) ,
  \Delta_N \omega^{n+1}  \right\rangle    \nonumber\\
  &\qquad\le
   \dt  \left\|  \nabla_N \left( \omega^{n+1} - \omega^n \right) \right\|_2^2
  + \frac12 \nu \dt  \left\|  \Delta \omega^{n+1}  \right\|_2^2 + C_{10} \dt
   \label{est-coll-H1-13}
\end{align}
with $C_{10} = C_7 C_6^{3/2} C_9$

  The third nonlinear term in (\ref{est-coll-H1-3-1}), (\ref{est-coll-H1-3-2})
can be analyzed in a similar way. We first look at
$\u^{n+1} \cdot \nabla_N \omega^{n+1}$.
A bound for $\| \u^{n+1} \|_{\infty}$ can be obtained in the same fashion
as (\ref{est-coll-H1-8}):
\begin{eqnarray}
  \left\|  \u^{n+1} \right\|_{\infty}
  &\le&
  C \left\|  \U^{n+1}  \right\|_{H^{1+\delta}}
  \le  C \left\|  \ppsi^{n+1}  \right\|_{H^{2+\delta}}
  \le  C \left\|  \pomega^{n+1}  \right\|_{H^{\delta}}   \nonumber
  \le  C \left\|  \pomega^{n+1}  \right\|^{1-\frac{\delta}{2}}_2
  \cdot \left\|  \pomega^{n+1} \right\|_{H^2}^{\frac{\delta}{2} }   \nonumber
\\
  &\le&
  C \left\|  \pomega^{n+1} \right\|^{1- \frac{\delta}{2}}_2
  \left\|  \Delta \pomega^{n+1}  \right\|^{\frac{\delta}{2} }_2
  \le C C_6^{1- \frac{\delta}{2} }
  \left\|  \Delta_N \omega^{n+1} \right\|^{\frac{\delta}{2} }_2.
  \label{est-coll-H1-14}
\end{eqnarray}
Its combination with (\ref{est-coll-H1-6}) shows that
\begin{eqnarray}
  \dt \left\langle \u^{n+1} \DOT \nabla_N \omega^{n+1} ,
  \Delta_N \omega^{n+1} \right\rangle
  &\le&
  \dt   \left\|  \u^{n+1} \right\|_{\infty}
    \cdot \left\|   \nabla_N \omega^{n+1}  \right\|_2
    \cdot \left\|   \Delta_N \omega^{n+1}  \right\|_2   \nonumber
\\
  &\le&
  C C_6^{3/2}   \dt
  \left\|  \Delta_N \omega^{n+1} \right\|_2^{\frac{3+\delta}{2} }  .  \label{est-coll-H1-15}
\end{eqnarray}
This analysis can be applied to the term $\nabla_N \cdot ( \u^{n+1} \omega^{n+1} )$
in the same way:
\begin{eqnarray}
   & & \left\|  \nabla_N \cdot \left(   \u^{n+1}  \omega^{n+1}  \right)  \right\|_2
   \le
      \left\|  {\cal D}_{Nx}  \left( u^{n+1} \omega^{n+1}  \right)  \right\|_2
 +   \left\|  {\cal D}_{Ny}  \left( v^{n+1} \omega^{n+1}  \right)  \right\|_2  \nonumber
\\
  & &  \qquad
   \le 2  \left( \left\|  \partial_x  \left( U^{n+1} \pomega^{n+1}   \right)  \right\|_{2}
   + \left\|  \partial_y  \left( V^{n+1}  \pomega^{n+1}   \right)  \right\|_{2} \right)  ,
   \label{est-coll-H1-15-2}
\\
  & &
   \left\|  \partial_x \left( U^{n+1}  \pomega^{n+1}  \right)  \right\|_{2}
   =  \left\|  U_x^{n+1}   \pomega^{n+1}
   +  U^{n+1} \pomega_x^{n+1}  \right\|_{2}    \nonumber
\\
  & &  \qquad
   \le \left\|  U_x^{n+1}   \right\|_{L^{2/(1-\delta)}}
   \cdot  \left\| \pomega^{n+1}  \right\|_{L^{2/\delta}}
   + \left\| U^{n+1}  \right\|_{L^\infty}
   \cdot  \left\| \pomega_x^{n+1}  \right\|_{2}    \nonumber
\\
  & &   \qquad
   \le C \left\|  U^{n+1}  \right\|_{H^{1+\delta}}
   \cdot  \left\| \nabla \pomega^{n+1}  \right\|_{2}
  \le C C_6^{3/2}
  \left\|  \Delta_N \omega^{n+1} \right\|_2^{\frac{1+\delta}{2} }  ,
   \label{est-coll-H1-15-3}
\\
  & &
  \left\|  \partial_y \left( V^{n+1}  \pomega^{n+1}  \right)  \right\|_{2}
  \le C C_6^{3/2}
  \left\|  \Delta_N \omega^{n+1} \right\|_2^{\frac{1+\delta}{2} }  .
   \label{est-coll-H1-15-4}
\end{eqnarray}
As a result, we arrive at the following estimate:
\begin{align}
  &\dt \left\langle \u^{n+1} \DOT \nabla_N \omega^{n+1}
  + \nabla_N \cdot \left(   \u^{n+1}  \omega^{n+1}  \right) ,
  \Delta_N \omega^{n+1} \right\rangle\notag\\
  &\hbox to120pt{} \le   C_{11} C_6^{3/2}   \dt
  \left\|  \Delta_N \omega^{n+1} \right\|_2^{\frac{3+\delta}{2} }  .  \label{est-coll-H1-15-5}
\end{align}
Again, since $\frac{3+\delta}{2} < 2$, we can apply Young's inequality and obtain
\begin{equation}
  \left\|  \Delta_N \omega^{n+1} \right\|_2^{\frac{3+\delta}{2} }
  \le  \frac{\nu}{2 C_{11} C_6^{3/2} } \left\|  \Delta_N \omega^{n+1} \right\|_2^2
   + C_{12} ,  \label{est-coll-H1-16}
\end{equation}
in which $C_{12}$ depends on $C_6$, $C_{11}$ and $\delta$. Going back  to
(\ref{est-coll-H1-15-5}), we have an estimate for the third nonlinear term:
\begin{align}
  &\dt \left\langle \u^{n+1} \DOT \nabla_N \omega^{n+1}
  + \nabla_N \cdot \left(   \u^{n+1}  \omega^{n+1}  \right) ,
  \Delta_N \omega^{n+1} \right\rangle\notag\\
  &\hbox to120pt{}\le   \frac12 \nu \dt  \left\|  \Delta_N \omega^{n+1} \right\|_2^2
   +  C_{13} \dt  ,    \label{est-coll-H1-17}
\end{align}
with $C_{13} = C_{12} C_{11} C_6^{3/2}$.

   Finally, a combination of (\ref{est-coll-H1-1})--(\ref{est-coll-H1-3-2}), (\ref{est-coll-H1-5}),
(\ref{est-coll-H1-5-2}), (\ref{est-coll-H1-13}) and (\ref{est-coll-H1-17}) results in
\begin{eqnarray}
  & &
  \|  \nabla_N \omega^{n+1}  \|_2^2  -  \|  \nabla_N \omega^n  \|_2^2
  +  \left( 1 - \left( 1 + \frac{C_{14} \tilde{C}_1^2}{\nu}  \right) \dt  \right)
  \|  \nabla_N \left( \omega^{n+1}  -  \omega^n \right) \|_2^2   \nonumber
\\
  & &
  +  \frac12 \nu \dt \|  \Delta_N \omega^{n+1}  \|_2^2
  \le \left(  \frac{2 M^2}{\nu} +  C_{10} + C_{13}  \right)  \dt  .  \label{est-coll-H1-18}
\end{eqnarray}
Under a constraint similar to (\ref{constraint-coll-dt-1}) and a trivial constraint
$\dt \le \frac14$ for the time step:
\begin{equation}
   \frac{C_{14} \tilde{C}_1^2}{\nu} \dt  \le \frac12 ,  \, \, \,
   \dt \le \frac14 ,  \quad \mbox{i.e.}, \quad
   \dt \le  \mbox{min}  \left( \frac{\nu}{2 C_{14} \tilde{C}_1^2}  , \frac14 \right) ,
   \label{constraint-coll-dt-2}
\end{equation}
we have
\begin{eqnarray}
   & &
   \|  \nabla_N \omega^{n+1}  \|_2^2  -  \|  \nabla_N \omega^n  \|_2^2
  +  \frac14  \left\|  \nabla_N \left( \omega^{n+1}  -  \omega^n \right)  \right\|_2^2
  +  \frac12 \nu \dt \|  \Delta_N \omega^{n+1}  \|_2^2
  \le  C_{15} \dt ,   \nonumber
\\
  & &
  \mbox{with}  \, \, \,
  C_{15} = \frac{2 M^2}{\nu} +  C_{10} + C_{13} .  \label{est-coll-H1-19}
\end{eqnarray}
Furthermore, an application of elliptic regularity
\begin{equation}
   \| \nabla_N \omega^{n+1} \|_2  \le C_{16}  \| \Delta_N \omega^{n+1} \|_2 ,
   \label{ellip regularity-coll}
\end{equation}
implies that
\begin{equation}
   \|  \nabla_N \omega^{n+1}  \|_2^2  -  \|  \nabla_N \omega^n  \|_2^2
  +  C_{17} \nu \dt \|  \nabla_N \omega^{n+1}  \|_2^2
  \le  C_{15} \dt ,  \quad  \mbox{with}  \, \, \,
  C_{17} = \frac{1}{2 C_{16}^2}.  \label{est-coll-H1-20}
\end{equation}
Applying an induction argument to the above estimate yields
\begin{eqnarray}
   & &
   \|  \nabla_N \omega^{n+1}  \|_2^2
   \le (1+ C_{17} \nu \Delta t)^{-(n+1)} \|  \nabla_N \omega^0  \|_2^2
   + \frac{C_{15} }{C_{17} \nu} ,  \quad \mbox{i.e.},  \, \, \,  \nonumber
\\
  & &
   \|  \nabla_N \omega^{n+1}  \|_2
   \le (1+ C_{17} \nu \Delta t)^{-\frac{n+1}{2} }  \|  \nabla_N \omega^0  \|_2
   + \sqrt{\frac{C_{15} }{C_{17} \nu}} := C_{18} .   \label{est-coll-H1-21}
\end{eqnarray}
Again, $C_{18}$ is a time dependent value; however, its time dependence is
in exponential decay so that a global in time bound is available.

  In addition, we also have the $L^2 (0, T; H^2)$ bound for the
numerical solution:
\begin{equation}
  \frac12 \nu \dt \sum_{k=i+1}^{N_k}   \left\|  \Delta_N  \omega^k  \right\|_2^2
  \le   \|  \nabla_N \omega^i  \|_2^2  +  C_{14}  \, \left( T^* - t^i \right) .  \label{est-coll-H1-22}
\end{equation}

\subsection{Recovery of the a-priori $H^\delta$ assumption (\ref{est-coll-a priori})}

With the $L^\infty (0,T; L^2)$ and $L^\infty (0,T; H^1)$ estimate for the numerical
vorticity solution, namely (\ref{est-coll-L2-11}) and  (\ref{est-coll-H1-21}), we are able to
recover the $H^\delta$ assumption (\ref{est-coll-a priori}):
\begin{eqnarray}
  \left\|  \omega^{n+1} \right\|_{H_h^\delta}
  &=&
  \left\|  \pomega^{n+1}  \right\|_{H^{\delta}}   \nonumber
  \le  C \left\|  \pomega^{n+1}  \right\|^{1-\delta}_2
  \cdot \left\|  \pomega^{n+1} \right\|_{H^1}^{\delta}   \nonumber
\\
  &\le&
  C_{\delta} \left\|  \pomega^{n+1} \right\|^{1- \delta}_2
  \left\|  \nabla \pomega^{n+1}  \right\|^{\delta }_2
  \le C_{\delta} C_6^{1- \delta}  C_{18}^{\delta}.
  \label{a priori-coll-1}
\end{eqnarray}
For simplicity, by taking $\delta=\frac12$, we see that  (\ref{est-coll-a priori}) is also
valid at time step $t^{n+1}$ if we set
\begin{equation}
  \tilde{C}_1 = C_{\delta} \sqrt{ C_6 C_{18} } .  \label{a priori-coll-2}
\end{equation}
Note that $C_6$ and $C_{18}$ are independent of $\tilde{C}_1$ in the derivation.
The constant $\tilde{C}_1$ is only used in the time step constraint
(\ref{constraint-coll-dt-1}). Therefore, an induction can be applied so that the a-priori
$H^\delta$ assumption (\ref{est-coll-a priori}) is valid at any time step under
a global time step constraint
\begin{equation}
   \dt  \le \frac{\nu}{4 C_{\delta}^2 C_6 C_{18} } .  \label{a priori-coll-3}
\end{equation}
Again, note that both $C_6$ and $C_{18}$ contain an exponential decay in time
and therefore are bounded by a given constant in time.

  In other words, under (\ref{a priori-coll-3}), a global in time constant constraint for the
time step, the proposed semi-implicit scheme
(\ref{scheme-coll-1st-1})--(\ref{scheme-coll-1st-3})
is unconditionally stable (in terms of spatial grid size and final time).
In addition, an asymptotic decay for the $L^2$ and $H^1$ norm for the vorticity
(equivalent to $H^1$ and $H^2$ norms for the velocity) can be derived.
Lemma~\ref{collocation} is proven.

\begin{appendix}

\section{A Wente type estimate}\label{s:wente}
     The goal here is to present a Wente type estimate that is applicable to our doubly periodic setting.
  Original estimate of the Jacobian term (essentially $H^{-1}$ norm) goes back to \cite{W1969}.
  Here we need an estimate on the $L^2$ norm of the Jacobian. The case with homogeneous Dirichlet boundary condition can be found in \cite{K2009,KT2000}.

\medskip
\begin{prop}\label{wente}
There exists an absolute constant $C_w\ge1$ such that
\begin{align}
   &\|\sgb\psi\cdot\gb\phi\|_{H^{-1}}^{} \le C_w\,\|\psi\|_{H^1}^{}\|\phi\|_{H^1}^{}
	& &\forall\>\psi\in\Hperz^1(\Omega), \>\phi\in\Hperz^1(\Omega)\label{q:wente1}\\
   &\|\sgb\psi\cdot\gb\phi\|_{2}^{} \le C_w\,\|\psi\|_{H^2}^{}\|\phi\|_{H^1}^{}
	& &\forall\>\psi\in\Hperz^2(\Omega), \> \phi\in\Hperz^1(\Omega)\label{q:wente2}\\
   &\|\sgb\psi\cdot\gb\phi\|_{2}^{} \le C_w\,\|\psi\|_{H^1}^{}\|\phi\|_{H^2}^{}
	& &\forall\psi\in\Hperz^1(\Omega), \> \phi\in\Hperz^2(\Omega).\label{q:wente3}
\end{align}
\end{prop}
\begin{proof}
Let $\Omega=(0,2\pi)^2$ as before and $\Omt:=(-2\pi,4\pi)^2$.
Let $\rho\in C_0^\infty(\Real^2)$ be such that $\rho=1$ in $\Omega$,
$\rho=0$ in $\Real^2-\Omt$ and $\rho(x)\in[0,1]$ for all $x\in\Real^2$.
Here $\psi$ and $\phi$ are $2\pi$-periodic functions on $\Real^2$.
The proof of \eqref{q:wente2} is based on Lemma~1 in
\cite{K2009}, which states that, in our notation, for $\rho\psi\in H_0^2(\Omt)$
and $\rho\phi\in H_0^1(\Omt)$, one has
\begin{equation}\label{q:wente4}
   \|\sgb(\rho\psi)\cdot\gb(\rho\phi)\|_{L^2(\Omt)}^{}
	\le C_K(\Omt)\,\|\rho\psi\|_{H^2(\Omt)}^{}\|\rho\phi\|_{H^1(\Omt)}^{}\,.
\end{equation}
Noting that
\begin{equation}\begin{aligned}
   \|\gb(\rho\psi)\|_{\Omt}^{}
	&= \|\gb(\rho\psi)\|_{\Omega}^{} + \|\gb(\rho\psi)\|_{\Omt-\Omega}^{}\\
	&\le \|\gb\psi\|_{\Omega}^{} + \|\rho\gb\psi\|_{\Omt-\Omega}^{} + \|\psi\gb\rho\|_{\Omt-\Omega}^{}\\
	&\le \|\gb\psi\|_{\Omega}^{} + \|\gb\psi\|_{\Omt-\Omega}^{} + \|\psi\|_{\Omt-\Omega}^{}\|\gb\rho\|_{L^\infty(\Omt-\Omega)}^{}\\
	&\le \|\gb\psi\|_{\Omega}^{} + 8\,\|\gb\psi\|_{\Omega}^{} + 8\,c_0^{}\|\gb\psi\|_{\Omt-\Omega}^{}\|\gb\rho\|_{L^\infty(\Omt-\Omega)}^{}\,,
\end{aligned}\end{equation}
and a similar computation for $\|\rho\psi\|_{H^2}^{}$,
the right-hand side of \eqref{q:wente4} is majorised as
\begin{equation}
   \|\rho\psi\|_{H^2(\Omt)}^{}\|\rho\phi\|_{H^1(\Omt)}^{}
	\le (9+8c_0^{}\,\|\gb\rho\|_{L^\infty(\Real^2)}^{})^2\,C_K(\Omt)^2\,\|\psi\|_{H^2(\Omega)}^{}\|\phi\|_{H^1(\Omega)}^{}\,.
\end{equation}
Since the left-hand side of \eqref{q:wente4} majorises
$\|\sgb\psi\cdot\gb\phi\|_{L^2(\Omega)}^{}$, \eqref{q:wente1} follows.
The proof of \eqref{q:wente3} is completely analogous,
using the estimate \cite[Lemma~1]{K2009},
\begin{equation}
   \|\sgb(\rho\psi)\cdot\gb(\rho\phi)\|_{L^2(\Omt)}^{} \le C_K\,\|\rho\psi\|_{H^1(\Omt)}^{}\|\rho\phi\|_{H^2(\Omt)}^{}\,.
\end{equation}
for $\rho\psi\in H_0^1(\Omt)$ and $\rho\phi\in H_0^2(\Omt)$.

For \eqref{q:wente1}, we take $w\in H^1_0 (\Omt)$ and compute
\begin{equation}\begin{aligned}
   \|\sgb\psi\cdot\gb\phi\|_{H^{-1}(\Omt)}^{}
	&= \sup_{\|w\|_{H^1(\Omt)}=1}\,(\sgb\psi\cdot\gb\phi,w)_{L^2(\Omt)}\\
	&\le \sup_{\|w\|_{H^1(\Omt)}=1}\,\|\gb\phi\|_{L^2(\Omt)}\|\gb\psi\|_{L^2(\Omt)}\|\gb w\|_{L^2(\Omt)}\\
	&= \|\gb\phi\|_{L^2(\Omt)}\|\gb\psi\|_{L^2(\Omt)}
\end{aligned}\end{equation}
where the inequality follows from (3.8) in \cite{W1969}.
Arguing as above, \eqref{q:wente1} follows.
\end{proof}

\section{A convergence result on long time behaviors}\label{s:conv}
    Here we present a modified version of the abstract result presented in \cite{W2010}, so that it is applicable to the current situation, where the phase space is only a subset of a Hilbert (or reflexive Banach) space.

\medskip
\begin{prop}\label{abs_conv_stat}
 Let $\{S(t)\}_{t\ge 0}$ be a continuous semi-group on a complete metric space $X$ which is a subset of a separable Hilbert space $H$ with the inherited distance $($norm$)$ $\|\cdot\|$.
Suppose that the semi-group generates a continuous dissipative dynamical system $($in the
sense of possessing a compact global attractor $\mathcal{A}${}$)$  on
 $X$. Let $\{S_k\}_{0<k\le k_0}$ be a family of continuous maps on
$X$ which generates a family of discrete dissipative dynamical
system $($with global attractor $\mathcal{A}_k${}$)$ on $X$. We further assume
 that the following two conditions
are satisfied.

\smallskip
\begin{itemize}
\item[H1:]$[$Uniform boundedness$]$ There exists a $k_1\in (0,k_0]$
such that $\{S_k\}_{0<k\le k_1}$ is uniformly bounded in the sense
that
\begin{equation}\label{u-bdd}
    K = \bigcup_{0<k\le k_1} \Attr_k
\end{equation}
is bounded in $X$.
\item[H2:]$[$Finite time uniform convergence$]$ $S_k$ uniformly
converges to $S$ on any finite time interval $($modulo any initial
layer$)$ and uniformly for initial data from the global attractor of
the scheme in the sense that there exists $t_0>0$ such that for any $T^*>t_0>0$
\begin{equation}
    \lim_{k\rightarrow 0}\sup_{\bu\in\Attr_k, nk\in [t_0, T^*]}\|S^n_k\bu -
S(nk)\bu\| = 0.
    \label{u-conv-long}
\end{equation}
\end{itemize}
   Then the global attractors converge in the sense of Hausdorff semi-distance, i.e.
\begin{equation}
  \lim_{k\rightarrow 0} \mbox{dist}_H(\mathcal{A}_k, \mathcal{A})=0.
\end{equation}
Moreover, if  the following three more stringent conditions are satisfied:

\smallskip
\begin{itemize}
\item[H3:]$[$Uniform dissipativity$]$ There exists a $k_1 \in (0,
k_0)$ such that $\{S_k\}_{0<k\le k_1}$ is uniformly dissipative in
the sense that
\begin{equation}
    K = \bigcup_{0<k\le k_1} \mathcal{A}_k \label{u-dissip}
\end{equation}
    is pre-compact in $X$.
\item[H4:]$[$Uniform convergence on the unit time interval$\,]$ $S_k$
uniformly converges to $S$ on the unit time interval $($modulo an
initial layer$)$ and uniformly for initial data from the global
attractor of $S_k$ in the sense that for any $t_0\in (0, 1)$
\begin{equation}
    \lim_{k\rightarrow 0}\sup_{\bu\in\mathcal{A}_k, nk\in [t_0, 1]}\|S^n_k\bu -
S(nk)\bu\| = 0.
    \label{u-conv}
\end{equation}
\item[H5:]$[$Uniform continuity of the continuous system$]$ $\{S(t)\}_{t\ge 0}$ is uniformly continuous on $K$  on the unit time
interval in the sense that for any $T^*\in [0,1]$
\begin{equation}
  \lim_{t\rightarrow T^*} \sup_{\bu\in K} \|S(t)\bu-S(T^*)\bu\| = 0,
  \label{u-cont}
\end{equation}
\end{itemize}
    then the invariant measures of the discrete dynamical system $\{S_k\}_{0<k\le k_0}$
converge to invariant measures of the continuous dynamical system
$S$. More precisely, let $\mu_k\in \mathcal{IM}_k$ where
$\mathcal{IM}_k$ denotes the set of all invariant measures of
$S_k$. There must exist a subsequence, still denoted $\{\mu_k\}$,
and $\mu\in \mathcal{IM}$ $($an invariant measure of $S(t)${}$)$, such
that $\mu_k$ weakly converges to $\mu$, i.e.,
\begin{equation}
    \mu_k \rightharpoonup \mu, \ \mbox{as}\ k\rightarrow 0.
\end{equation}
\end{prop}

\begin{proof}
The proof is exactly the same as those in \cite{W2010, W2009}. We leave the detail to the interested reader.
\end{proof}

\end{appendix}

\section*{Acknowledgement}
This work is supported in part by grants from the National Science
Foundation (DMS1008852 for XW, DCNS0959382 for SG and CW), AFSOR (FA-9550-09-0208 for SG, 10418149 for SG and CW),  a Modern Applied Mathematics 111 project at Fudan University from the Chinese MOE (for XW), and a COFRS fund from FSU (for XW).

\end{document}